\documentclass[onefignum,onetabnum]{siamart190516}

\usepackage{amsmath,amssymb,amsxtra,amsfonts}
\usepackage{mathrsfs}
\usepackage{graphicx}
\usepackage{float,verbatim}
\usepackage{threeparttable,booktabs}
\usepackage{underscore}
\usepackage{algorithm}
\usepackage{subfigure}
\usepackage{bm}
\usepackage{extarrows}
\usepackage{multirow}
\usepackage{makecell}
\usepackage[noend]{algpseudocode}
\usepackage{color}
\usepackage{xcolor}
\usepackage{hyperref}

\graphicspath{{./pic/}}

\newcommand\rmd {\,{\rm d}}
\newtheorem{remark}{Remark}[section]
\newtheorem{example}{Example}[section]
\newtheorem{assumption}{Assumption}[section]

\algdef{SE}[DOWHILE]{Do}{doWhile}{\algorithmicdo}[1]{\algorithmicwhile\ #1}

\allowdisplaybreaks

\title{Conditional Stability and Numerical Reconstruction of a Parabolic Inverse Source Problem Using Carleman Estimates\thanks{The work of B. Jin is supported by Hong Kong RGC General Research Fund (Project 14306824), Hong Kong RGC  ANR / RGC Joint
Research Scheme (A-CUHK402/24) and a sart-up fund from The Chinese University of Hong Kong. The work of Z. Zhou is supported by National Natural Science Foundation of China (Project 12422117), Hong Kong
Research Grants Council (15303122) and an internal grant of Hong Kong Polytechnic University (Project
ID: P0038888, Work Programme: 1-ZVX3).}}

\author{Tianhao Hu\thanks{Department of Mathematics, The Chinese University of Hong Kong, Shatin, New Territories, Hong Kong SAR, P.R. China (\texttt{thhu@link.cuhk.edu.hk, qmquan@cuhk.edu.hk, quanqm@whu.edu.cn, b.jin@cuhk.edu.hk})}\and Xinchi Huang\thanks{Department of Physics, The University of Tokyo, 7-3-1, Hongo, Bunkyo-ku, Tokyo, 113-8654, Japan, and
Quemix Inc., 2-11-2, Nihombashi Chuo-ku, Tokyo, 103-0027, Japan} \and Bangti Jin\footnotemark[2] \and Qimeng Quan\footnotemark[2] \and Zhi Zhou\thanks{Department of Applied Mathematics, The Hong Kong Polytechnic University, Kowloon, Hong Kong, P.R. China (\texttt{zhizhou@polyu.edu.hk})}}

\date{\today}

\begin{document}

\maketitle

\maketitle
\begin{abstract}
In this work we develop a new numerical approach for recovering a spatially dependent source component in a standard parabolic equation from partial interior measurements. We establish novel conditional Lipschitz stability and H\"{o}lder stability for the inverse problem with and without boundary conditions, respectively, using suitable Carleman estimates. Then we propose a numerical approach for solving the inverse problem using conforming finite element approximations in both time and space. Moreover, by utilizing the conditional stability estimates, we prove rigorous error bounds on the discrete approximation. We present several numerical experiments to  illustrate the effectiveness of the approach. \vskip5pt

\end{abstract}

\pagestyle{myheadings}
\thispagestyle{plain}

\section{Introduction}\label{sec:intro}
In this work, we investigate the inverse problem of recovering a space-dependent source component $f$ in parabolic equations. Let $\Omega\subset \mathbb{R}^n$ ($n=1, 2, 3$) be a bounded space domain with a smooth boundary $\partial \Omega$ and fix the terminal time $T>0$. Consider the following parabolic governing equation
\begin{equation}\label{eqn:gov}
	 \partial_t u + A u = F, \quad \text{in } \Omega\times (0,T),
\end{equation}
with a separable source $F=R(x,t)f(x)$. The space-dependent elliptic operator $A$ is given by
\begin{equation*}
	A v := - \sum_{i,j=1}^n \partial_{x_i}(a_{ij}\partial_{x_j} v) + \sum_{j=1}^n b_j \partial_{x_j} v + cv,
	\quad \forall v\in H^2(\Omega),
\end{equation*}
with the coefficients $b_1,\dots, b_n,c\in L^\infty(\Omega)$ and the coefficient matrix  $\big[a_{ij}\big]\in W^{1,\infty}(\Omega;\mathbb{R}^{n\times n})$ satisfying
\begin{equation*}
	\mu |\xi|^2 \leq \sum_{i,j=1}^n a_{ij}\xi_i\xi_j \leq \mu^{-1} |\xi|^2,
	\quad\mbox{with}\ \xi = (\xi_1,\ldots,\xi_n)^T\in \mathbb{R}^n, \ \forall x\in\Omega,
\end{equation*}
for some constant $\mu\in(0,1)$. For any given $f\in L^2(\Omega)$, equation \eqref{eqn:gov} generally lacks uniqueness when both initial condition $u(0)\equiv u_0$ and boundary condition $u|_{\partial\Omega\times (0,T]}$ are missing. Additional measurements are required in order to guarantee the well-definedness of the inverse problem.

The concerned inverse problem is to (approximately) recover the space-dependent component $f^\dag$ in the source $F$ from suitable over-posed data. In this work, we focus on partial interior data. Specifically, let $\omega$ be a proper compact subdomain of $\Omega$ and $I_\zeta(t_0)\subset(0,T)$ be a time subinterval centered at the point $t=t_0$ with a radius $\zeta\in(0,\min(T-t_0,t_0))$. Consider the following two types of noisy observational data:
\begin{equation}\label{data: zero boundary}
\mbox{(when }u^\dagger = 0 \mbox{ on } \partial\Omega\times (0,T))\quad
\left\{\begin{aligned}
&(q,\partial_t q) \approx (u^\dagger,\partial_t u^\dagger),\ \mbox{in}\ \omega\times I_{\zeta}(t_0),\\
&p \approx Au^\dagger(t_0),\ \mbox{in }\ \Omega,
\end{aligned}
\right.
\end{equation}
\begin{equation}\label{data: unknown boundary}
\mbox{(when } u^\dag \mbox{ is unknown on } \partial\Omega\times (0,T))\quad\left\{\begin{aligned}
& (q,\partial_t q) \approx (u^\dagger,\partial_t u^\dagger), \ \mbox{in}\ \omega\times I_{\zeta}(t_0),\\
& r\approx\nabla\partial_t u^\dagger, \ \mbox{in}\ \omega\times I_{\zeta}(t_0),\\
&p \approx Au^\dagger(t_0),\ \mbox{in }\ \Omega,
\end{aligned}
\right.
\end{equation}
where $u^\dag$ denotes the exact state (corresponding to the exact source $f^\dag$).
The accuracy $\delta$ of the observational data \eqref{data: zero boundary} and \eqref{data: unknown boundary} are given by the noise level $\delta$, defined respectively by
\begin{equation*}
\delta:=\left\{\begin{aligned}
& \|q-u^\dagger\|_{H^1(I_\zeta(t_0);L^2(\omega))} + \|p-Au^\dagger(t_0)\|_{L^2(\Omega)},\quad &\mbox{for } \eqref{data: zero boundary}, \\
& \|q-u^\dagger\|_{H^1(I_\zeta(t_0);L^2(\omega))} + \|p-Au^\dagger(t_0)\|_{L^2(\Omega)} + \|r - \nabla\partial_t u^\dagger\|_{L^2(I_\zeta(t_0);L^2(\omega))}, \quad &\mbox{for } \eqref{data: unknown boundary}.
\end{aligned}
\right.
\end{equation*}
The inverse problem is to recover $f^\dag$ from the given observational data.
Inverse source problems for parabolic equations of recovering an unknown source component (in either time or space or both) in a parabolic PDE from limited measurement data arise in a broad range of practical applications, e.g., heat conduction \cite{WoodburyBeck:2023}, contaminant transport \cite{el2002inverse,Andrle03042015} and biomedical imaging \cite{KLOSE2005323}. This class of inverse problems is often ill-posed due to a lack of existence, uniqueness and stability, which poses significant challenges to their stable numerical reconstruction.

In this work, we contribute to the mathematical analysis and computational techniques of the inverse source problem for the model \eqref{eqn:gov}, and make the following two contributions. First, under mild (regularity) conditions on $f$ and $R\equiv R(x,t)$ in Assumption \ref{ass:reg}, we establish two novel conditional stability estimates with / without boundary conditions:
\begin{align*}
& \mbox{Lipschitz type: }\quad\|f\|_{L^2(\Omega)}\leq c(\|\partial_tu\|_{L^2(\omega\times I_\zeta(t_0))} + \|Au(t_0)\|_{L^2(\Omega)}), \\
& \mbox{H\"{o}lder type: }\quad \|f\|_{L^2(\Omega_0)}\leq c(\|\partial_tu\|_{L^2(I_\zeta(t_0); H^1(\omega))}+\|Au(t_0)\|_{L^2(\Omega)})^\theta,
\end{align*}
where $\Omega_0$ is an arbitrary subdomain such that $\overline{\Omega}_0\subset\Omega$, and the constants $c>0$ and $\theta\in(0,1)$ depend on $\Omega_0,\omega,t_0,\zeta,\Omega$ and $R$.
See Theorems \ref{thm: Lip-stab} and \ref{thm:hold-stab} for the precise statements. The key of the proof lies in constructing suitable weighted functions and two crucial Carleman estimates; see Lemmas \ref{lem: Carleman estimate} and \ref{lem:local-car-estimate} and the proof of Theorem \ref{thm:hold-stab} for details. Moreover, these stability results can be extended to the perturbation form, which plays a central role in the error analysis. Second, inspired by the conditional stability, we propose a novel least-squares formulation based on the residual of the governing equation and an $L^2(\Omega)$ penalty on the source $f$, and develop a fully discrete space-time FEM scheme for the inverse  problem. We approximate the state $u$ and the source $f$ in the least-squares formulation using the space-time $H^2$ conforming FE space (cf. \eqref{eqn:tensor element}), and the standard $H^1(\Omega)$ conforming linear FE space, respectively. By combining the stability estimates in Theorems \ref{thm:Lip-perturb-stability} and \ref{thm:hold-perturb-stab} and consistency bounds on the fully discrete scheme in Lemmas \ref{lem:approx-uhs-fhs-Ihpih} and \ref{lem:approx-uhs-fhs-Ihpih-hol}, we derive novel \textit{a priori} error bounds in $L^2(\Omega)$ and $L^2(\Omega_0)$ for the recovered source $f_h^*$; see Theorem \ref{thm:err-estimate} for the precise estimates. Moreover, we carry out several numerical experiments in one- and two-dimensions to illustrate the performance of the method. The numerical results agree well with theoretical convergence rates.

Now we review existing studies from the theoretical and numerical aspects on parabolic inverse source problems. There are several theoretical studies on the inverse source problem \cite{Cannon:1968,yamamoto1993conditional,choulli2004conditional,imanuvilov1998lipschitz}. Yamamoto  \cite{yamamoto1993conditional} studied the inverse source problem for parabolic equations in rectangular domains with one single lateral boundary measurement over whole / partial time intervals, and established the uniqueness and logarithmic conditional stability for a class of unknown sources lying in a bounded set in some Sobolev space; see also \cite{choulli2004conditional} for doubly logarithmic stability for \textit{a posteriori} boundary measurement. Imanuvilov and Yamamoto \cite{imanuvilov1998lipschitz} investigated the inverse source problem on a general domain $\Omega$ and established Lipschitz stability results for two types of overdetermined data, in which the key tool is Carleman estimates cf. \cite[Lemma 1.1]{emanuilov1995controllability}.
This approach has been substantially further developed for the parabolic model. Imanuvilov and Yamamoto  \cite{Imanuvilov2001Carleman} established pointwise Carleman estimates, and proved a novel stability estimate in a weak space (i.e., $H^{-1}(\Omega)$). For more extensive discussions on Carleman estimates for parabolic problems and their applications, we refer readers to  the comprehensive survey \cite{yamamoto2009carleman}. Recently, Huang et al \cite{huang2020stability} studied the inverse source problem for a parabolic equation without the \textit{a priori} knowledge of the initial condition, and proved a H\"{o}lder stability estimate from partial boundary observations using a novel local Carleman estimate \cite[Lemma 2]{huang2020stability}.

In practice, one often employs a suitable regularization strategy and constructs a discrete approximation. This is often achieved by combining Tikhonov regularization \cite{EnglHankeNeubauer:1996,ItoJin:2015} with classic / novel discretization methods, e.g., the finite difference method (FDM), the Galerkin finite element method (FEM) and deep neural networks (DNNs). H\`{a}o et al \cite{hao2017determination} investigated Tikhonov regularization for recovering the source from  a finite number of the integral means around discrete points in the domain $\Omega$, and also proved the convergence of the Galerkin FEM discretization. For the inverse source problem for the diffusion equation from the terminal measurement at scattered points subject to random noise, Chen et al \cite{chen2022stochastic} established stochastic error bounds in weaker topologies for both continuous regularized and its FEM approximations. Zhang et al \cite{zhang2023stability} employed physics informed neural network to recover the source and provided an \textit{a priori} error estimate for the neural network approximation. The error analysis in \cite{chen2022stochastic} and \cite{zhang2023stability} utilizes the stability estimate of the continuous problem with one single measurement over the whole domain $\Omega$. In sharp contrast, so far there are very few works on the numerical analysis of the inverse source problem with partial interior data. The challenge lies in the fact that continuous stability estimates typically are derived via Carleman estimates which involve high regularity assumptions and thus are not directly amenable with numerical treatment. For discrete models, one possible strategy is to establish discrete Carleman estimates \cite{boyer2010discrete1d,boyer2010discretehd,boyer2011uniform}. Boyer et al. \cite{boyer2010discrete1d,boyer2010discretehd} employed a finite difference method to discretize the parabolic problem in the space variable in the context of uniform null-controllability of parabolic equations, and derived a global discrete Carleman estimate for approximating the elliptic operator \cite[Theorem 5.5]{boyer2010discrete1d} and \cite[Theorem 1.4]{boyer2010discretehd}. This requires delicate mesh construction, discrete calculus and the assumption on the large Carleman parameter $s$ and the mesh size $h$. An alternative approach is to utilize standard discretization schemes that is compatible with the continuous Carleman estimates. Burman et al. \cite{burman2018fully} developed a fully discrete scheme using continuous linear finite elements in space and backward Euler discretization in time for parabolic data assimilation problems, and for the piecewise linear temporal interpolation of the time-stepping solutions, and derived an \textit{a priori} error bound using the continuous Carleman estimates.
The approach only requires the FE space be consistent with the regularity involved in the Carleman estimates.

The rest of the paper is organized as follows. In Section \ref{sec:stability}, we present two conditional stability estimates for the concerned inverse problem. In Section \ref{sec:approximation}, we develop the numerical scheme and derive error bounds on the discrete approximation. In Section \ref{sec:numer}, we present several numerical experiments in one- and two-dimensional cases to illustrate the numerical scheme.
For any $m\geq0$ and $r\geq1$, we denote by $W^{m,r}(\Omega)$ and $W^{m,r}_0(\Omega)$ the standard Sobolev spaces of order $m$, equipped with the norm $\|\cdot\|_{W^{m,r}(\Omega)}$ \cite{Adams2003Sobolev}. Further, we write $H^{m}(\Omega)$ and $H^{m}_{0}(\Omega)$ with the norm $\|\cdot\|_{H^m(\Omega)}$ if $r=2$ and write $L^r(\Omega)$ with the norm $\|\cdot\|_{L^r(\Omega)}$ if $m=0$. The notation $(\cdot,\cdot)$ denotes the $L^2(\Omega)$ inner product. We also use Bochner spaces: for a Banach space $B$, let
$W^{m,r}(0,T;B) = \{v: v(t)\in B\ \mbox{for a.e.}\ t\in(0,T)\ \mbox{and}\ \|v\|_{W^{m,r}(0,T;B)}<\infty \}$.
We write the time subinterval $I\equiv I_\zeta(t_0)$, and  denote the space-time domains by $Q=\Omega\times I$ and $Q_\omega=\omega\times I$. Moreover, we denote by $c$, with or without a subscript, a generic constant which may differ at each occurrence but is always independent of the discretization parameters $h$ and $\tau$, the noise level $\delta$, and the regularization parameter $\gamma$. Throughout,  $\nabla=(\partial_{x_1},\partial_{x_2},\dots,\partial_{x_n})$ denotes the space gradient and $\nabla_{x,t} := (\nabla, \partial_t)$.

\section{Conditional stability}\label{sec:stability}
In this section we recall preliminaries on Carleman estimates and derive two new conditional stability estimates for the concerned inverse source problem.
\subsection{Carleman estimates}
First we recall several useful global and interior Carleman estimates. Note that there exists a smooth function $d\in C^2(\overline\Omega)$ and a constant $\sigma_d$ such that \cite[Lemma 4.1]{yamamoto2009carleman}:
\begin{equation}\label{d:condition}
d>0 \quad \text{in }\Omega,
	\qquad d=0 \quad \text{on }\partial\Omega,
	\qquad |\nabla d| \ge \sigma_d >0 \quad \text{in } \overline{\Omega\setminus \omega}.
\end{equation}
Then we define two singular functions by $\alpha \equiv\alpha(x,t)$ and $\varphi\equiv\varphi(x,t)$ in the space-time domain $\overline{\Omega}\times I$ by
\begin{equation}\label{eqn:al-varphi}
   \alpha(x,t) = \frac{e^{\lambda d(x)}-e^{2\lambda\|d\|_{C(\overline\Omega)}}}{(t-t_0+\zeta)(t_0+\zeta-t)} \quad\mbox{and}\quad \varphi(x,t) = \frac{e^{\lambda d(x)}}{(t-t_0+\zeta)(t_0+\zeta-t)},
\end{equation}
for some large $\lambda>0$. Next we define a regular function by $\psi\equiv\psi(x,t)$ in $\overline{Q}$ by
\begin{equation}\label{eqn:psi}
    \psi(x,t) = d(x) - \beta(t-t_0)^2,
\end{equation}
with the constant $\beta>0$ chosen sufficiently large such that $\|d\|_{C(\overline\Omega)}-\beta\zeta^2<0$. Moreover, we define $\vartheta\equiv \vartheta(x,t):= e^{\lambda \psi}$ in the cylinder $\overline{Q}$.

The following lemma collects several pointwise estimates on the functions $\alpha$, $\varphi$ and $\psi$.
\begin{lemma}\label{lem:point-estimate}
    Let $\alpha$, $\varphi$ and $\psi$ be given by \eqref{eqn:al-varphi} and \eqref{eqn:psi}. Then the following estimates hold
    \begin{equation*}
    \begin{split}
        &\alpha < 0,\quad e^{2s\alpha}\leq e^{2s\alpha(t_0)}\leq1,\quad \partial_t\alpha \leq c\zeta\varphi^2, \quad \varphi^{-1}\leq \zeta^2\quad \mbox{in }\ \overline{\Omega}\times I, \\
        & \psi(t_0)>0 \ \quad\mbox{in }\ \Omega,\quad \psi(x,t) \leq \psi(x,t_0) \quad\mbox{in }\ \overline{Q},\\& \psi(x,t) \leq 0 \quad\mbox{on }\ \partial\Omega\times I, \quad \psi(x,t_0\pm\zeta)<0\quad\mbox{in }\ \Omega.
    \end{split}
    \end{equation*}
    Moreover, for any $s>0$ and $k\in \mathbb{N}$, the following asymptotic property holds
\begin{equation*}
    \lim_{t\to t_0\pm \zeta} \varphi^ke^{s\alpha} = 0, \quad\forall x\in\overline\Omega.
\end{equation*}
\end{lemma}
\begin{proof}
    By direct computation and the construction of the function $d$, we have
    \begin{equation*}
        \partial_t\alpha = \frac{-(2t_0-2t)}{(t-t_0+\zeta)^2(t_0+\zeta-t)^2}\big(e^{\lambda d(x)}-e^{2\lambda\|d\|_{C(\overline\Omega)}}\big)\leq c\zeta\varphi^2,
    \end{equation*}
     and $\varphi^{-1} = (t-t_0+\zeta)(t_0+\zeta-t)e^{-\lambda d(x)} \le \zeta^2$. The asymptotic property holds by the the exponential decay of $e^{s\alpha}$ and the remaining estimates follow directly from the definitions.
\end{proof}

Let $Lv:= \partial_tv + Av$. The next lemma gives one standard global Carleman estimate \cite[Theorem 4.1]{yamamoto2009carleman}.
\begin{lemma}\label{lem: Carleman estimate}
Let $v\in H^1(I;L^2(\Omega))\cap L^2(I;H^2(\Omega)\cap H_0^1(\Omega))$. Then there exist $\lambda_0>0$ and $s_0>0$ such that the following estimate holds for all $\lambda\geq\lambda_0$ and  $s\ge s_0$
\begin{align*}
	&\int_Q  \bigg[\frac{1}{s\varphi}  (|\partial_t v|^2 + |Av|^2) + s\lambda^2 \varphi |\nabla v|^2
	+ s^3\lambda^4\varphi^3 |v|^2\bigg] e^{2s\alpha} \ {\rm d}x{\rm d}t \\
	\leq &c \int_Q |Lv|^2 e^{2s\alpha} \ {\rm d}x{\rm d}t
	+ c\int_{Q_\omega} s^3\lambda^4\varphi^3 |v|^2 e^{2s\alpha} \ {\rm d}x{\rm d}t.
\end{align*}
\end{lemma}

The next result provides an interior Carleman estimate.  Unlike Lemma \ref{lem: Carleman estimate}, the following statement does not require any boundary data on the function $v$. Thus it is called an interior Carleman estimate.
\begin{lemma}\label{lem:local-car-estimate}
     Let $v\in H^1(I;L^2(\Omega))\cap L^2(I;H^2(\Omega))$. Then there exist $\lambda_0>0$ and $s_0>0$ such that the following estimate holds for all $\lambda\geq\lambda_0$ and  $s\ge s_0$
    \begin{align*}
    &\int_Q  \bigg[\frac{1}{s\vartheta}  (|\partial_t v|^2 + |Av|^2) + s\lambda^2 \vartheta |\nabla v|^2 + s^3\lambda^4\vartheta^3 |v|^2\bigg] e^{2s\vartheta} \ {\rm d}x{\rm d}t \\ & \leq c \int_Q |Lv|^2 e^{2s\vartheta} \ {\rm d}x{\rm d}t + c\int_{Q_\omega} \big(s^3\lambda^4\vartheta^3 |v|^2 +  s\lambda^2\vartheta |\nabla v|^2\big)e^{2s\vartheta} \ {\rm d}x{\rm d}t \\&\quad + c \int_{\partial \Omega \times I}s^3\lambda^4\vartheta^3\left(\left|\nabla_{x, t} v\right|^2+|v|^2\right) e^{2 s \vartheta} {\rm d}x{\rm d}t \\& \quad + c \int_{\Omega}s^3\lambda^4\vartheta^3(x,t_0\pm\zeta)\left(\left|\nabla v\left(x, t_0\pm\zeta\right)\right|^2+\left|v\left(x, t_0\pm\zeta\right)\right|^2\right)e^{2 s \vartheta\left(x, t_0\pm\zeta\right)} {\rm d}x.
\end{align*}
\end{lemma}
\begin{proof}
Consider first the case $v\in C_0^\infty(Q)$. Let $w = ve^{s\vartheta}$ and $\sigma:=\sum_{i, j=1}^n a_{i j}\left(\partial_{x_i} d\right)\left(\partial_{x_j} d\right)$. Then with $w\in C_0^\infty(Q)$,  repeating the argument of \cite[pp. 7-16]{yamamoto2009carleman} leads to
{\small \begin{equation}\label{estimate-w}
\begin{split}
& \int_{Q} s^3 \lambda^4 \vartheta^3 \sigma^2 w^2 {\rm d} x {\rm d} t + \int_{Q} s \lambda^2 \vartheta \sigma \sum_{i, j=1}^n a_{i j}\left(\partial_{x_i} w\right)\left(\partial_{x_j} w\right) {\rm d} x {\rm d} t + \left(\varepsilon-c\lambda^{-1}\right) \int_{Q} \frac{1}{s \vartheta}\left|\partial_t w\right|^2 {\rm d} x {\rm d} t \\
&  \leq  c \int_{Q} |Lv|^2 \mathrm{e}^{2 s \vartheta} {\rm d} x {\rm d} t\!+\!  c\int_{Q} (s\lambda \vartheta\! + \!\lambda^2)|\nabla w|^2 {\rm d} x {\rm d} t\! \\
 & \quad +\! c \int_{Q}\left(s^3 \lambda^3 \vartheta^3+s^2 \lambda^4 \vartheta^2\right) |w|^2 {\rm d} x {\rm d} t \!+ \!c\epsilon\!\int_Q s \lambda^2 \vartheta |\nabla w|^2 \rmd x \rmd t,
\end{split}
\end{equation}}
\noindent with $\epsilon\in(0,\frac14]$ to be chosen later. See Appendix \ref{append} for the details.
Using the positivity of the coefficient matrix $[a_{ij}]_{n\times n}$ and $|\nabla d|\geq \sigma_d>0$ in $\overline{\Omega\backslash\omega}$, cf. \eqref{d:condition}, we obtain
{\footnotesize\begin{equation*}
\begin{aligned}
& \mu^2\sigma_d^4\int_{\Omega\backslash\omega\times I} s^3 \lambda^4 \vartheta^3  |w|^2 {\rm d} x {\rm d} t + (\mu^2\sigma_d^2 - c\epsilon)\int_{\Omega\backslash\omega\times I} s \lambda^2 \vartheta |\nabla w|^2 {\rm d} x {\rm d} t + \left(\varepsilon-c\lambda^{-1}\right) \int_{Q} \frac{1}{s \vartheta}\left|\partial_t w\right|^2 {\rm d} x {\rm d} t \\
\leqslant & c \int_{Q} |Lv|^2 \mathrm{e}^{2 s \vartheta} {\rm d} x {\rm d} t+  c\int_{Q} (s\lambda \vartheta + \lambda^2)|\nabla w|^2 {\rm d} x {\rm d} t+ c \int_{Q}\left(s^3 \lambda^3 \vartheta^3+s^2 \lambda^4 \vartheta^2\right) |w|^2 {\rm d} x {\rm d} t +  \!c\epsilon\!\int_{Q_\omega} s \lambda^2 \vartheta |\nabla w|^2 \rmd x \rmd t
\end{aligned}
\end{equation*}}
Fix $\epsilon = \min(\frac{\mu^2\sigma_d^2}{2c},\frac14)$ and take $\lambda$ large enough such that $\epsilon - c\lambda^{-1}\geq c_{\epsilon}>0$. By supplementing the region $(\Omega\backslash\omega) \times I$ into $Q$ and taking $s>0$ large enough, we get
\begin{equation*}
\begin{aligned}
& \int_{Q} s^3 \lambda^4 \vartheta^3  |w|^2 {\rm d} x {\rm d} t + \int_{Q} s \lambda^2 \vartheta |\nabla w|^2 {\rm d} x {\rm d} t +  \int_{Q} \frac{1}{s \vartheta}\left|\partial_t w\right|^2 {\rm d} x {\rm d} t \\
\leqslant&  c \int_{Q} |Lv|^2 \mathrm{e}^{2 s \vartheta} {\rm d} x {\rm d} t + c\int_{Q_\omega} s^3\lambda^4\vartheta^3 |w|^2 +  s\lambda^2\vartheta |\nabla w|^2 \ {\rm d}x{\rm d}t,
\end{aligned}
\end{equation*}
Using the definition of $w = ve^{s\vartheta}$ and the identity $Av = Lv - \partial_t v$, we can directly derive the desired estimate for the case $v \in C_0^{\infty}(Q)$. If $v \notin C_0^{\infty}(Q)$, the boundary integrals $v(\cdot, t_0 \pm \zeta)$ and $\nabla v(\cdot, t_0 \pm \zeta)$ over $\Omega$, and $v$ and $\nabla_{x,t} v$ on the parabolic boundary $\partial \Omega \times I$ should be retained in the formula. Hence the desired estimate follows.
\end{proof}

\subsection{Conditional stability}
Now we establish two useful conditional stability estimates for the concerned inverse source problem under the following mild regularity assumption.
\begin{assumption}\label{ass:reg}
    $f\in L^2(\Omega)$ and $R\in W^{1,\infty}(I;L^\infty(\Omega))$, and the condition $R(x,t_0) \ge r_0 >0$ holds in $\Omega$.
\end{assumption}

First, we state the following Lipschitz-type stability.
\begin{theorem}\label{thm: Lip-stab}
Let Assumption \ref{ass:reg} hold.	Let $u\in H^2(I;L^2(\Omega))\cap H^1(I;H^2(\Omega)\cap H_0^1(\Omega))$ be the solution to problem \eqref{eqn:gov}. Then there exists a constant $c>0$, depending on $\omega,t_0,\zeta,\Omega, R$ and $A$, such that
\begin{equation*}
\|f\|_{L^2(\Omega)}\leq c(\|\partial_tu\|_{L^2(Q_\omega)} + \|Au(t_0)\|_{L^2(\Omega)}).
\end{equation*}
\end{theorem}
\begin{proof}
Let $v = \partial_t u$. Then it satisfies $v=0$ on $\partial\Omega\times I$ and
\begin{equation}\label{eqn:Lip-gov}
\left\{
\begin{aligned}
 \partial_t v + A v &= \partial_t R f, \quad \text{in}\ Q,\\
 v(t_0) &= R(t_0)f - Au(t_0), \quad \text{in}\ \Omega.
\end{aligned}
\right.
\end{equation}
Using the triangle inequality for the second line in \eqref{eqn:Lip-gov}, we get
\begin{align*}
\int_\Omega |R(t_0)f|^2 e^{2s\alpha(t_0)} \ {\rm d}x
&\leq c\bigg(\int_\Omega |Au(t_0)|^2 e^{2s\alpha(t_0)} \ {\rm d}x
+ \int_\Omega |v(t_0)|^2 e^{2s\alpha(t_0)} \ {\rm d}x\bigg)=:c({\rm I} + {\rm II}).
\end{align*}
It remains to estimate ${\rm II}$. Since the property $\lim_{t\to t_0-\zeta} e^{2s\alpha(t)} = 0$ holds in $\Omega$, it follows that
\begin{equation*}
   {\rm II} = \int_\Omega \int_{t_0-\zeta}^{t_0} \partial_t \big(|v|^2 e^{2s\alpha}\big) \ {\rm d}x{\rm d}t = 2 \int_\Omega\int_{t_0-\zeta}^{t_0} (v \partial_t v + s\partial_t\alpha |v|^2)e^{2s\alpha} \ {\rm d}x{\rm d}t.
\end{equation*}
Then for a large $s>0$, by Lemma \ref{lem:point-estimate} and  Young's inequality $ab\leq \frac{\epsilon^{-1}}{2} a^2 + \frac{\epsilon}{2} b^2$ (for every $a,b,\epsilon\in\mathbb{R}_+$), there holds
	\begin{align*}
		{\rm II}
		&\leq c \int_Q\left(\frac{1}{s^2\varphi^2}|\partial_t v|^2 + s^2\varphi^2|v|^2
		+ s\varphi^2 |v|^2\right)e^{2s\alpha} \ {\rm d}x{\rm d}t\\
		&\leq cs^{-1}\int_Q \left(\frac{1}{s\varphi}|\partial_t v|^2 + s^3\lambda^4\varphi^3|v|^2\right)e^{2s\alpha} \ {\rm d}x{\rm d}t.
	\end{align*}
By Lemma \ref{lem: Carleman estimate} and Assumption \ref{ass:reg}, we obtain
\begin{align*}
{\rm II}
&\leq cs^{-1}\int_Q |\partial_t R|^2 |f|^2e^{2s\alpha} \ {\rm d}x{\rm d}t
+ {c\int_{Q_\omega} s^3\lambda^4\varphi^3 |v|^2 e^{2s\alpha} \ {\rm d}x{\rm d}t} \\
&\leq cs^{-1}\int_Q |f|^2e^{2s\alpha(t_0)}\ {\rm d}x + ce^{cs}\|v\|_{L^2(Q_\omega)}^2.
\end{align*}
This and the positivity condition in Assumption \ref{ass:reg} imply
\begin{equation*}
    \int_\Omega |f|^2 e^{2s\alpha(t_0)} \ {\rm d}x \leq cs^{-1}\int_Q |f|^2e^{2s\alpha(t_0)}\ {\rm d}x + ce^{cs}(\|v\|_{L^2(Q_\omega)}^2 + \| Au(t_0)\|_{L^2(\Omega)}^2).
\end{equation*}
Let $s>0$ be large enough such that $cs^{-1} \leq \frac12$. Thus with the definition of $v$, we arrive at the desired stability estimate.
\end{proof}

The next theorem presents a conditional  H\"{o}lder stability, which does not require the knowledge of the boundary condition of the solution $u$.
\begin{theorem}\label{thm:hold-stab}
Let Assumption \ref{ass:reg} hold. Let $u\in H^2(I;H^1(\Omega))\cap H^1(I;H^2(\Omega))$ be the solution to equation \eqref{eqn:gov} with a uniform a priori bound
\begin{equation}\label{ineq:priori-bound-u}
\|u\|_{H^2(I;H^1(\Omega))} +  \|u\|_{H^1(I;H^2(\Omega))}\leq M.
\end{equation}
Then for any $\Omega_0\subset\subset \Omega$ (i.e., $\overline{\Omega_0} \subset \Omega$), there exists constants $c>0$ and $\theta\in (0,1)$,
depending on $\Omega_0,\omega,t_0,\zeta,\Omega, R$ and $A$, such that
\begin{equation*}
\|f\|_{L^2(\Omega_0)}\leq cM^{1-\theta}(\|Au(t_0)\|_{L^2(\Omega)} + \|\partial_tu\|_{L^2(I; H^1(\omega))})^\theta.
\end{equation*}
\end{theorem}
\begin{proof}
Let $v = \partial_t u$. Then it satisfies
\begin{equation}\label{eqn:holder-gov}
\left\{
\begin{aligned}
 \partial_t v + A v &= \partial_t R f, \quad \text{in}\ Q,\\
 v(t_0) &= R(t_0)f - Au(t_0), \quad \text{in}\ \Omega.
\end{aligned}
\right.
\end{equation}
Using Lemma \ref{lem: Carleman estimate} for the governing equation in \eqref{eqn:holder-gov}, we obtain
\begin{equation*}
   \int_Q \bigg(\frac{1}{s\vartheta}|\partial_t v|^2
+ s^3\lambda^4\vartheta^3 |v|^2\bigg) e^{2s\vartheta} \ {\rm d}x{\rm d}t \leq {\rm I} + {\rm II} + {\rm III} + {\rm IV},
\end{equation*}
with the four terms given by
\begin{align*}
& {\rm I}= c \int_Q |\partial_t R|^2|f|^2 e^{2s\vartheta} \ {\rm d}x{\rm d}t, \\
& {\rm II} = c\int_{Q_\omega} \big(s^3\lambda^4\vartheta^3 |v|^2 +  s\lambda^2\vartheta |\nabla v|^2\big)e^{2s\vartheta} \ {\rm d}x{\rm d}t, \\
& {\rm III} = c \int_{\partial \Omega \times I}s^3\lambda^4\vartheta^3(\left|\nabla_{x, t} v\right|^2+|v|^2) e^{2 s \vartheta} {\rm d}\sigma{\rm d}t, \\
& {\rm IV} = c \int_{\Omega}s^3\lambda^4\vartheta^3\left(x, t_0\pm\zeta\right)(\left|\nabla v\left(x, t_0\pm\zeta\right)\right|^2+\left|v\left(x, t_0\pm\zeta\right)\right|^2)e^{2 s \vartheta\left(x, t_0\pm\zeta\right)} {\rm d}x.
\end{align*}
By the argument in Theorem \ref{thm: Lip-stab} and Young's inequality, we arrive at the following estimate
\begin{align*}
\int_\Omega |v(t_0)|^2 e^{2s\vartheta(t_0)} \ {\rm d}x
&=  \int_\Omega \int_{t_0-\zeta}^{t_0} \partial_t \big(|v|^2 e^{2s\vartheta}\big) \ {\rm d}x{\rm d}t + \int_\Omega |v(t_0-\zeta)|^2 e^{2s\vartheta(t_0-\zeta)} \ {\rm d}x \\
& \leq \int_\Omega \int_{t_0-\zeta}^{t_0} (2v\partial_t v + 2s\partial_t\vartheta |v|^2)e^{2s\vartheta} \ {\rm d}x{\rm d}t + s^{-1}{\rm IV} \\
& \leq cs^{-1}\int_Q \big(\frac{1}{s\vartheta}|\partial_t v|^2
+ s^3\lambda^4\vartheta^3 |v|^2\big) e^{2s\vartheta} \ {\rm d}x{\rm d}t + s^{-1}{\rm IV} \\& \leq cs^{-1}( {\rm I} + {\rm II} + {\rm III} + {\rm IV} ).
\end{align*}
This estimate, the second line in \eqref{eqn:holder-gov} and the triangle inequality lead to
\begin{align*}
    \int_\Omega |R(t_0)|^2|f|^2e^{2s\vartheta(t_0)} \ {\rm d}x &\leq   c\left(\int_\Omega |v(t_0)|^2 e^{2s\vartheta(t_0)} \ {\rm d}x + \int_\Omega |Au(t_0)|^2e^{2s\vartheta(t_0)} \ {\rm d}x\right) \\
    & \leq cs^{-1}( {\rm I} + {\rm II} + {\rm III} + {\rm IV} ) + c\int_\Omega |Au(t_0)|^2e^{2s\vartheta(t_0)} \ {\rm d}x.
\end{align*}
Under Assumption \ref{ass:reg}, there holds for a sufficiently large $s>0$,
\begin{equation*}
    \int_{\Omega}|f|^2e^{2s\vartheta(t_0)}\ {\rm d}x  \leq cs^{-1}( {\rm II} + {\rm III} + {\rm IV} ) + c\int_\Omega |Au(t_0)|^2e^{2s\vartheta(t_0)} \ {\rm d}x.
\end{equation*}
By the choice of $\beta>0$, we have $d(x) - \beta\zeta^2<0$. Then with  $\sigma_0:=\min_{x\in\overline{\Omega_0}}\vartheta(t_0)>1$, we have
\begin{equation*}
   \sigma_1:=\max\Big\{\max_{(x,t)\in \partial\Omega\times I}\vartheta,\ \max_{x\in \Omega}\vartheta(t_0-\zeta) \Big\}\leq 1 < \sigma_0.
\end{equation*}
Thus by the estimate \eqref{ineq:priori-bound-u} and the trace inequality, there holds
\begin{equation*}
    cs^{-1}{\rm III} + cs^{-1}{\rm IV} \leq cs^2e^{2s\sigma_1}M^2.
\end{equation*}
Consequently, the following estimate holds for some $s_0>0$,
\begin{equation*}
    e^{2s\sigma_0}\|f\|_{L^2(\Omega_0)}^2 \leq cs^2e^{2s\sigma_1}M^2 + ce^{cs}\big(\|Au(t_0)\|_{L^2(\Omega)}^2 + \|v\|^2_{L^2(I; H^1(\omega))} \big), \quad\forall s>s_0.
\end{equation*}
Let $\mu = \sigma_0-\sigma_1 >0$. Shifting $s$ by $s_0$ and noting the inequality $\sup_{s>0}(s+s_0)^2e^{-c(s+s_0)}<\infty$ yield
\begin{equation*}
    \|f\|_{L^2(\Omega_0)} \leq ce^{-c\mu s}M + ce^{cs}\big(\|Au(t_0)\|_{L^2(\Omega)} + \|v\|_{L^2(I; H^1(\omega))}\big), \quad\forall s>0.
\end{equation*}
When $M \leq \|Au(t_0)\|_{L^2(\Omega)} + \|v\|_{L^2(I; H^1(\omega))}$, the desired estimate holds directly. Otherwise, balancing these two terms and using the definition of $v$ give the stability estimate.
\end{proof}

\subsection{Stability estimate with perturbation}

The next result gives the stability estimate in the presence of a perturbation, which will be needed in the error analysis of the numerical scheme in Section \ref{sec:numer}.
\begin{theorem}\label{thm:Lip-perturb-stability}
Let Assumption \ref{ass:reg} hold and $G\in H^1(I;L^2(\Omega))$.
Let $u\in H^2(I;L^2(\Omega))\cap H^1(I;H^2(\Omega)\cap H_0^1(\Omega))$ be the solution to problem \eqref{eqn:gov} with the source $F=Rf + G$. Then there exists a constant $c>0$, depending on $\omega,t_0,\zeta,\Omega, R$
and $A$, such that
\begin{equation*}
	\|f\|_{L^2(\Omega)}\le c\big(\|\partial_tu\|_{L^2(Q_\omega)} + \|Au(t_0)\|_{L^2(\Omega)}
	+ \|G\|_{H^1(I;L^2(\Omega))}\big).
\end{equation*}
\end{theorem}
\begin{proof}
The argument is similar to Theorem \ref{thm: Lip-stab}, and hence we only sketch the proof. Let $v=\partial_t u$. Then it satisfies $v=0$ on $\partial\Omega\times I$ and
\begin{equation}\label{eqn:Lip-pertub-gov}
\left\{
\begin{aligned}
& \partial_t v + A v = \partial_t R f + \partial_t G, && \text{in}\ Q,\\
& v(t_0) = R(t_0)f + G(t_0) - Au(t_0), && \text{in}\ \Omega.
\end{aligned}
\right.
\end{equation}
Thus by the continuous embedding $H^1(I)\hookrightarrow C(\overline I)$, we obtain
\begin{align*}
&\int_\Omega |R(t_0)f|^2 e^{2s\alpha(t_0)} \ {\rm d}x
\\ \leq & c\bigg(\int_\Omega |Au(t_0)|^2 e^{2s\alpha(t_0)} \ {\rm d}x
+ \int_\Omega |G(t_0)|^2 e^{2s\alpha(t_0)} \ {\rm d}x + \int_\Omega |v(t_0)|^2 e^{2s\alpha(t_0)} \ {\rm d}x\bigg) \\
\leq & c\big(\|Au(t_0)\|_{L^2(\Omega)} + \|G\|_{H^1(I;L^2(\Omega))}^2\big) + c\int_\Omega |v(t_0)|^2 e^{2s\alpha(t_0)} \ {\rm d}x.
\end{align*}
By repeating the argument in Theorem \ref{thm: Lip-stab}, for the term ${\rm I}:= \int_\Omega |v(t_0)|^2 e^{2s\alpha(t_0)} \ {\rm d}x$, we get
\begin{align*}
{\rm I} &\leq cs^{-1}\int_Q |\partial_t R|^2|f|^2e^{2s\alpha}\ {\rm d}x{\rm d}t + cs^{-1}\int_Q |\partial_t G|^2 e^{2s\alpha}\ {\rm d}x{\rm d}t+ c\int_{Q_\omega} s^3\lambda^4\varphi^3 |v|^2 e^{2s\alpha} \ {\rm d}x{\rm d}t   \\& \leq cs^{-1}\int_\Omega |f|^2e^{2s\alpha(t_0)}\ {\rm d}x +  ce^{cs}(\|G\|_{H^1(I,L^2(\Omega))} + \|v\|_{L^2(Q_\omega)}^2).
\end{align*}
By taking $s>0$ large enough and noting the positivity condition in Assumption \ref{ass:reg}, we arrive at the desired estimate.
\end{proof}

The following theorem states the H\"{o}lder stability with a perturbation term $G$. The proof is similar to Theorems \ref{thm:hold-stab} and \ref{thm:Lip-perturb-stability}, and hence omitted.
\begin{theorem}\label{thm:hold-perturb-stab}
Let Assumption \ref{ass:reg} hold and $G\in H^1(I;L^2(\Omega))$. Let $u\in H^2(I;H^1(\Omega))\cap H^1(I;H^2(\Omega))$ be the solution to problem \eqref{eqn:gov} with the source $F = Rf + G$. Moreover, assume that the following a priori bound holds
\begin{equation*}
\|u\|_{H^2(I;H^1(\Omega))} +  \|u\|_{H^1(I;H^2(\Omega))}\leq M.
\end{equation*}
Then for any $\Omega_0\subset  \subset \Omega$, there exists constants $c>0$ and $\theta\in (0,1)$,
depending on $\Omega_0,\omega,t_0,\zeta,\Omega, R$ and $A$, such that
\begin{equation*}
\|f\|_{L^2(\Omega_0)}\leq cM^{1-\theta}(\|Au(t_0)\|_{L^2(\Omega)} + \|\partial_tu\|_{L^2(I; H^1(\omega))} + \|G\|_{H^1(I;L^2(\Omega))})^\theta.
\end{equation*}
\end{theorem}

\section{Numerical approximation and error analysis}\label{sec:approximation}
Inspired by the conditional stability estimates in Section \ref{sec:stability}, we now construct a numerical approximation using the Galerkin FEM \cite{Thome2006GalerkinFE} in both space and time, and provide rigorous error bounds on the discrete approximation. The error estimates are consistent with the stability estimates in Section \ref{sec:stability}.

\subsection{Regularized problem and its numerical approximation}
First we develop a numerical procedure for recovering the source $f$. For both space and time discretizations, we employ the Galerkin FEM. Let $\Omega\subset \mathbb{R}^n$ be a  simply connected convex polyhedral domain with a boundary $\partial\Omega$. Let $h\in(0,h_0]$ for some $h_0<1$ and  $\mathcal{T}_h:=\cup\{K_j\}_{j=1}^{N_h}$ be a shape regular quasi-uniform simplicial triangulation of the domain $\Omega$ into mutually disjoint open face-to-face
subdomains $K_j$, such that $\Omega:= {\rm Int}(\cup_j\{\overline{K}_j\})$. On the triangulation $\mathcal{T}_h$, we define two $H^2(\Omega)$ conforming finite element spaces $V_h$ and $\mathring{V}_h:= V_h\cap H_0^1(\Omega)$ (e.g., Hermite elements and Argyris elements for $n=1,2$ \cite{BrennerScott:book2008} and Zhang elements for $n=3$ \cite{zhang2009family}). Moreover, we define the continuous linear element space $W_h$ by
\begin{equation*}
	W_{h}:=	\{w_{h}\in H^1(\Omega):	w_{h}|_{K}\in P_1(K),\quad\forall K\in\mathcal{T}_h\},
\end{equation*}
with $P_r(K)$ being the space of polynomials of degree at most $r$ over $K$.
For the FE spaces $\mathring{V}_h$ and $W_h$, we define the $L^2(\Omega)$ projections $P_h: L^2(\Omega)\mapsto \mathring{V}_h$ and $\pi_h: L^2(\Omega) \mapsto W_h$, respectively, by
\begin{align*}
	(P_hv, \varphi_h) &= (v,\varphi_h),\quad\forall v\in L^2(\Omega),\ \varphi_h\in \mathring{V}_h, \\ (\pi_hw, \psi_h) &= (w,\psi_h),\quad\forall w\in L^2(\Omega),\ \psi_h\in W_h.
\end{align*}
Then the following stability and approximation properties of $P_h$ hold for $m,k\in\mathbb{N}$ \cite[p. 123]{BrennerScott:book2008}:
\begin{align}
	\label{ineq:Ph stab}
	\|P_h v\|_{H^m(\Omega)} &\le c\|v\|_{H^m(\Omega)},
	 \quad \forall v\in H^m(\Omega)\cap H_0^1(\Omega),\ 1\leq m\leq 2,\\
	\label{ineq:Ph approx}
	\|v - P_h v\|_{H^k(\Omega)} &\le ch^{m-k}\|v\|_{H^m(\Omega)},\quad
	 \forall v\in H^m(\Omega)\cap H_0^1(\Omega),\ 1\leq k\leq m\leq3.
\end{align}
Moreover, the following approximation property of $\pi_h$ holds for $m=1,2$  \cite[p. 82]{Braess2007}:
\begin{equation}\label{ineq:pih stab and approx}
	\|v - \pi_h v\|_{L^2(\Omega)} + h\|v - \pi_h v\|_{H^1(\Omega)} \le ch^m\|v\|_{H^m(\Omega)}, \quad\forall v\in H^m(\Omega).
\end{equation}
Now we describe the time discretization scheme. We divide the interval $I$ into $N$ uniform subintervals $I_n:=[t_n,t_{n+1})$, $n=1,\dots,N$, with the step size $\tau=|I|/N$ and the grid points $t_n=t_0-\zeta + (n-1)\tau$, $n=1,\dots,N+1$. Then we define the Hermite element space $V_\tau$ by
\begin{equation*}
V_\tau := \{v_\tau\in H^2(I): v_\tau|_{I_n}\in P_3(I_n) \}.
\end{equation*}
Let $P_\tau:L^2(I)\mapsto V_\tau$ be the $L^2(I)$ projection. Then the operator $P_\tau$ satisfies for $m,k\in\mathbb{N}$ \cite[p. 123]{BrennerScott:book2008}:
\begin{align}
	\label{ineq:Ptau stab}
	\|P_\tau v\|_{H^m(I)} &\le c\|v\|_{H^m(I)},
	\quad \forall v\in H^m(I),\ 1\leq m\leq 2,\\
	\label{ineq:Ptau approx}
	\|v - P_\tau v\|_{H^k(I)} &\le c\tau^{m-k}\|v\|_{H^m(I)},
	\quad \forall v \in H^m(I),\ 1\leq k\leq m\leq3.
\end{align}
Lastly, we construct the space-time $H^2$ conforming element spaces $V_{h,\tau}$ and $\mathring{V}_{h,\tau}$ using the tensor product:
\begin{equation}\label{eqn:tensor element}
\begin{split}
V_{h,\tau} := V_h\otimes V_\tau\subset H^2(I;H^2(\Omega))\quad \mbox{and}\quad
\mathring{V}_{h,\tau} := \mathring{V}_h\otimes V_\tau\subset H^2(I;H^2(\Omega)\cap H_0^1(\Omega)).
\end{split}
\end{equation}
The state $u$ and the source $f$ are discretized by the FE spaces $V_{h,\tau}$ and $W_h$, respectively. To develop a numerical scheme for recovering the source $f$, we employ the regularized least-square formulation with the equation residual. This leads to the following regularized functional
\begin{align}\label{eqn:full-reg prob}
\min_{u_{h,\tau}\in \mathring{V}_{h,\tau}, f_h\in W_h}J_{h,\tau}(u_{h,\tau},f_h) :=&
\frac12 \|u_{h,\tau}-q\|_{H^1(I;L^2(\omega))}^2
+ \frac12 \|A u_{h,\tau}(t_0) - p\|_{L^2(\Omega)}^2 \\&
+ \frac{1}{2} \|L u_{h,\tau} - Rf_h\|_{H^1(I;L^2(\Omega))}^2.	\nonumber
\end{align}
Similarly, we also define an alternative regularized functional when $u$ is not known on the boundary $\partial\Omega$:
{\small \begin{align}\label{eqn:full-reg prob holder}
\min_{u_{h,\tau}\in V_{h,\tau}, f_h\in W_h}\widetilde J_{\widetilde \gamma,h,\tau}(u_{h,\tau},f_h) :=&
		\frac12 \|u_{h,\tau}-q\|_{H^1(I;L^2(\omega))}^2 + \frac12 \|\nabla u_{h,\tau} - r\|_{H^1(I;L^2(\omega))}^2
		\\& +  \frac12 \|A u_{h,\tau}(t_0) - p\|_{L^2(\Omega)}^2
		+ \frac12 \|L u_{h,\tau} - Rf_h\|_{H^1(I;L^2(\Omega))}^2
	\nonumber\\&	+ \frac{\widetilde{\gamma}_f}{2} \|f_h\|^2_{L^2(\Omega)}  +\frac{\widetilde{\gamma}_u}{2}(\|u_{h,\tau} \|_{H^1(I;H^2(\Omega))}^2 + \|u_{h,\tau}\|_{H^2(I;H^1(\Omega))}^2),\nonumber
\end{align}}
with $\widetilde\gamma = (\widetilde{\gamma}_f,\widetilde\gamma_u)$, where $\widetilde{\gamma}_f$ and $\widetilde\gamma_u$ are positive scalars. Note that in the loss $\widetilde{J}_{\widetilde{\bm\gamma},h,\tau}$, we also include  an $L^2(\Omega)$ penalty on the source $f$.

The next result gives the unique solvability of problem \eqref{eqn:full-reg prob}. The identity \eqref{eqn:Euler-Lag} indicates that the reconstruction scheme involves only solving one linear system.
\begin{theorem}\label{thm:uni-sol-Lip-prob}
Let Assumption \ref{ass:reg} hold. Problem \eqref{eqn:full-reg prob} admits a unique solution $(u_{h,\tau}^*,f_h^*)\in \mathring{V}_{h,\tau}\times W_h$. Moreover, let the bilinear form $b\big[\cdot,\cdot\big]: (\mathring{V}_{h,\tau}\times W_h)^2 \to \mathbb{R}$ be defined by
\begin{align*}
b\big[(u_{h,\tau},f_h),(v_{h,\tau},g_h)\big] := & (u_{h,\tau}, v_{h,\tau})_{H^1(I;L^2(\omega))}
+ (Au_{h,\tau}(t_0),Av_{h,\tau}(t_0)) \\& + (Lu_{h,\tau}-Rf_h, Lv_{h,\tau}-Rg_h)_{H^1(I;L^2(\Omega))}.
\end{align*}
Then the minimizer $(u_{h,\tau}^*,f_h^*)$ satisfies
\begin{equation}\label{eqn:Euler-Lag}
b\big[(u_{h,\tau}^*,f_h^*),(v_{h,\tau},g_h)\big] = \big(q,v_{h,\tau}\big)_{H^1(I;L^2(\omega))} + \big(p,Av_{h,\tau}(t_0)\big),\quad\forall(v_{h,\tau},g_h)\in \mathring{V}_{h,\tau}\times W_h.
\end{equation}
\end{theorem}
\begin{proof}
By the standard argument in calculus of variation, the equivalence of the minimizer to problem \eqref{eqn:full-reg prob} and the solution to equation \eqref{eqn:Euler-Lag} follows. It remains to prove the unique solvability of \eqref{eqn:Euler-Lag}.  Let $|\!|\!|\cdot|\!|\!|$ be the energy semi-norm induced by the bilinear form $b$, i.e.,
\begin{equation}\label{eqn:energy norm}
|\!|\!|(v_{h,\tau},g_h)|\!|\!|: = \sqrt{b\big[(v_{h,\tau},g_h),(v_{h,\tau},g_h)\big]}.
\end{equation}
Below we prove that problem \eqref{eqn:Euler-Lag} has one unique zero solution in $\mathring{V}_{h,\tau}\times W_h$ when the data $p=q\equiv0$. By the definition of $b$, we deduce
\begin{equation*}
\left\{\begin{aligned}
v_{h,\tau} &= \partial_tv_{h,\tau} = 0\quad\mbox{in } Q_\omega,\\
Av_{h,\tau}(t_0) &= 0\quad \mbox{in }\Omega,\\  Lv_{h,\tau} &= Rg_h\quad\mbox{in } Q.
\end{aligned}\right.
\end{equation*}
Hence Theorem \ref{thm: Lip-stab} implies that $g_h\equiv0$ in $\Omega$ and $Lv_{h,\tau} = 0$ in $Q$. Next using the identity $v_{h,\tau} =0$ in $ Q_\omega$ and the Carleman estimate in Lemma \ref{lem: Carleman estimate} give $v_{h,\tau}=0$ in $Q$. This completes the proof of the theorem.
\end{proof}

The following theorem gives the existence and uniqueness of a minimizer to problem \eqref{eqn:full-reg prob holder}. The proof is similar to Theorem \ref{thm:uni-sol-Lip-prob}. The last two terms in \eqref{eqn:full-reg prob holder} imply that there exists a unique zero solution to the Euler-Lagrange equation when $p,q = 0$ and $r = \bm 0$.
\begin{theorem}
Problem \eqref{eqn:full-reg prob holder} admits a unique solution $(\widetilde{u}_{h,\tau}^*,\widetilde{f}_h^*)\in V_{h,\tau}\times W_h$. Moreover, let $\widetilde{b}\big[\cdot,\cdot\big]$ be defined by
\begin{align*}
\widetilde b\big[(u_{h,\tau},f_h),(v_{h,\tau},g_h)\big] &:= (u_{h,\tau},v_{h,\tau})_{H^1(I;L^2(\omega))} + (\nabla u_{h,\tau},\nabla v_{h,\tau})_{H^1(I;L^2(\omega))}
\\&\qquad + (A u_{h,\tau}(t_0),A v_{h,\tau}(t_0))_{L^2(\Omega)} + (
Lu_{h,\tau} - Rf_h,Lv_{h,\tau} - Rg_h)_{H^1(I;L^2(\Omega))}
	\\&\qquad	+ \widetilde{\gamma}_f (f_h,g_h)  +\widetilde{\gamma}_u\big((u_{h,\tau},v_{h,\tau})_{H^1(I;H^2(\Omega))} + (u_{h,\tau},v_{h,\tau})_{H^2(I;H^1(\Omega))}\big)
\end{align*}
over the space $(V_{h,\tau}\times W_h)^2$.
Then the minimizer $(\widetilde{u}_{h,\tau}^*,\widetilde{f}_h^*)$ satisfies
{\footnotesize\begin{equation*}
\widetilde b\big[(\widetilde{u}_{h,\tau}^*,\widetilde{f}_h^*),(v_{h,\tau},g_h)\big] = \big(q,v_{h,\tau}\big)_{H^1(I;L^2(\omega))} + (r,\nabla v_{h,\tau})_{H^1(I;L^2(\omega))} + \big(p,Av_{h,\tau}(t_0)\big),\quad\forall(v_{h,\tau},g_h)\in V_{h,\tau}\times W_h.
\end{equation*}}
\end{theorem}

\subsection{Error analysis}
Now we establish \textit{a priori} $L^2(\Omega)$ error estimates on the minimizers $f_h^*\in W_h$ to problem \eqref{eqn:full-reg prob} and $\widetilde{f}_h^*\in W_h$ to problem \eqref{eqn:full-reg prob holder} under the following regularity assumption on the exact source $f^\dag$ and exact state $u^\dag$.

\begin{assumption}\label{assum: dis-reg}
Assumption \ref{ass:reg} holds. Moreover, $ f^\dagger\in H^{1+r}(\Omega)$ and $u^\dagger\in H^{3+r}(I;L^2(\Omega))\cap H^{2+r}(I;H^2(\Omega))\cap H^1(I; H^{3+r}(\Omega))$ holds for $r=0$ or $1$ .
\end{assumption}

The next lemma provides an \textit{a priori} bound on $J_{h,\tau}(u_h^*,f_h^*)$.

\begin{lemma}\label{lem:approx-uhs-fhs-Ihpih}
Under Assumptions \ref{ass:reg} and \ref{assum: dis-reg}, the following estimate holds
\begin{equation*}
J_{h,\tau}(u^*_{h,\tau},f^*_h) \leq c(\delta^2 + \tau^{2+2r} + h^{2+2r}).
\end{equation*}
\end{lemma}
\begin{proof}
By the minimizing property of $(u^*_{h,\tau},f^*_h)$ and noting $(P_hP_\tau u^\dagger,\pi_h f^\dagger)\in \mathring{V}_{h,\tau}\times W_h$, we obtain
\begin{equation*}
    J_{h,\tau}(u^*_{h,\tau},f^*_h) \leq  J_{h,\tau}(P_hP_\tau u^\dagger,\pi_h f^\dagger) = {\rm I} + {\rm II} + {\rm III},
\end{equation*}
where the three terms $\rm I$, $\rm II$ and $\rm III$ are given by
\begin{equation}\label{ineq:error3term}
\begin{aligned}
&{\rm I}= \tfrac12 \|P_hP_\tau u^\dagger-q\|_{H^1(I;L^2(\omega))}^2, \
{\rm II} = \tfrac12 \|A P_hP_\tau u^\dagger(t_0) - p\|_{L^2(\Omega)}^2,\\ &
{\rm III} = \tfrac{1}{2} \|L P_hP_\tau u^\dagger - R\pi_h f^\dagger\|_{H^1(I;L^2(\Omega))}^2.
\end{aligned}
\end{equation}
By the triangle inequality, the approximation property \eqref{ineq:Ph approx} of $P_h$, and the $H^1(\Omega)$ stability \eqref{ineq:Ptau stab} and the approximation property \eqref{ineq:Ptau approx} of $P_\tau$, we have
\begin{align*}
    {\rm I} &\leq c\big(\|P_hP_\tau u^\dagger-P_\tau u^\dagger\|_{H^1(I;L^2(\omega))}^2 + \|P_\tau u^\dagger - u^\dagger
    \|_{H^1(I;L^2(\omega))}^2 + \| u^\dagger-
    q\|_{H^1(I;L^2(\omega))}^2 \big) \\& \leq c\big(h^{6+2r}\|u^\dagger\|^2_{H^1(I;H^{3+r}(\Omega))} + \tau^{4+2r}\|u^\dagger\|^2_{H^{3+r}(I;L^2(\Omega))}+ \delta^2) \leq c(h^{6+2r} + \tau^{4+2r} +\delta^2 ).
\end{align*}
Using the Sobolev embedding $H^1(I)\hookrightarrow C(\overline I)$, and the $H^2(\Omega)$ stability \eqref{ineq:Ph stab} and the approximation property \eqref{ineq:Ph approx} of  $P_h$, we deduce
\begin{align*}
    {\rm II} &\leq c\big(\|A P_hP_\tau u^\dagger(t_0) - A P_h u^\dagger(t_0)\|_{L^2(\Omega)}^2 + \|A P_h u^\dagger(t_0) - Au^\dagger(t_0)\|_{L^2(\Omega)}^2 + \|Au^\dagger(t_0) - p\|_{L^2(\Omega)}^2 \big) \\
    &\leq c\big(\|P_\tau u^\dagger -  u^\dagger\|_{H^1(I;H^2(\Omega))}^2 + \|P_hu^\dagger - u^\dagger\|_{H^1(I;H^2(\Omega))}^2 + \|Au^\dagger(t_0) - p\|_{L^2(\Omega)}^2 \big)  \\
    & \leq c\big( \tau^{2+2r}\|u^\dagger\|^2_{H^{2+r}(I;H^2(\Omega))}  + h^{2+2r}\|u^\dagger\|^2_{H^1(I; H^{3+r}(\Omega))} + \delta^2)\leq c(\tau^{2+2r} + h^{2+2r} + \delta^2).
\end{align*}
Next we bound the term ${\rm III}$. Since the equation $Lu^\dagger = Rf^\dagger$ holds in $Q$, we arrive at
\begin{align*}
    {\rm III} &\leq c(\|L P_hP_\tau u^\dagger - L P_\tau u^\dagger\|_{H^1(I;L^2(\Omega))}^2 + \|L P_\tau u^\dagger - Lu^\dagger\|_{H^1(I;L^2(\Omega))}^2 \\&\qquad+  \|Rf^\dagger - R\pi_h f^\dagger\|_{H^1(I;L^2(\Omega))}^2) =: {\rm III}_1 + {\rm III}_2 + {\rm III}_3.
\end{align*}
Below we bound the three terms $({\rm III}_i)_{i=1}^3$ separately. Assumption \ref{assum: dis-reg} and the estimate \eqref{ineq:pih stab and approx} lead to
$${\rm III}_3 \leq ch^{2+2r}\|f^\dag\|_{H^{1+r}(\Omega)}^2.$$
By the definition of the operator $L=\partial_t + A$ and the estimates \eqref{ineq:Ph stab}-\eqref{ineq:Ptau approx}, there holds
\begin{align*}
    {\rm III}_2 &\leq c(\|\partial_t P_\tau u^\dagger - \partial_tu^\dagger\|_{H^1(I;L^2(\Omega))}^2 + \|A P_\tau u^\dagger - Au^\dagger\|_{H^1(I;L^2(\Omega))}^2) \\
    & \leq c(\tau^{2+2r}\|u^\dagger\|^2_{H^{3+r}(I;L^2(\Omega))} + \tau^{2+2r}\|u^\dagger\|^2_{H^{2+r}(I;H^2(\Omega))}) \leq c\tau^{2+2r}.
\end{align*}
Similarly, the bound on the term ${\rm III}_1$ follows from the estimates \eqref{ineq:Ph stab}-\eqref{ineq:Ptau approx}:
\begin{align*}
    {\rm III}_1 & \leq c(\|\partial_t P_hP_\tau u^\dagger -  \partial_tP_\tau u^\dagger\|_{H^1(I;L^2(\Omega))}^2 + \|A P_hP_\tau u^\dagger -  AP_\tau u^\dagger\|_{H^1(I;L^2(\Omega))}^2) \\&
    \leq c(h^4\|u^\dagger\|^2_{H^2(I;H^2(\Omega))} + h^{2+2r} \|u^\dagger\|^2_{H^1(I;H^{3+r}(\Omega))}) \leq ch^{2+2r}.
\end{align*}
Combining the preceding estimates yields the desired assertion.
\end{proof}

The next lemma gives an \textit{a priori} estimate on $\widetilde J_{\widetilde \gamma,h,\tau}(\widetilde{u}^*_{h,\tau},\widetilde{f}^*_h)$.
\begin{lemma}\label{lem:approx-uhs-fhs-Ihpih-hol}
Let Assumption \ref{assum: dis-reg} hold. Then there holds
\begin{equation*}
\widetilde J_{\widetilde \gamma,h,\tau}(\widetilde{u}^*_{h,\tau},\widetilde{f}^*_h) \leq c(\widetilde{\gamma}_f + \widetilde{\gamma}_u + \delta^2 + \tau^{2+2r} + h^{2+2r}).
\end{equation*}
\end{lemma}
\begin{proof}
The argument is similar to Lemma \ref{lem:approx-uhs-fhs-Ihpih}.
Using the minimizing property of $(\widetilde{u}_{h,\tau}^*, \widetilde{f}_h^*)\in V_{h,\tau}\times W_h$ to problem \eqref{eqn:full-reg prob holder} and the estimates \eqref{ineq:Ph stab}-\eqref{ineq:Ptau approx}, there holds
\begin{align*}
    &\|\nabla P_hP_\tau u^\dagger - r\|_{H^1(I;L^2(\omega))}^2 \\\leq& c(\|\nabla P_hP_\tau u^\dagger - \nabla P_\tau u^\dagger\|_{H^1(I;L^2(\omega))}^2 + \|\nabla P_\tau u^\dagger -\nabla u^\dagger\|_{H^1(I;L^2(\omega))}^2 + \|\nabla u^\dagger - r\|_{H^1(I;L^2(\omega))}^2 ) \\ \leq& c(h^{4+2r}\|u^\dagger\|^2_{H^1(I;H^{3+r}(\Omega))} + \tau^{2+2r}\|u^\dagger\|^2_{H^{2+r}(I;H^1(\Omega))} + \delta^2).
\end{align*}
Thus Assumption \ref{assum: dis-reg} gives the estimate
$$ \|\nabla P_hP_\tau u^\dagger - r\|_{H^1(I;L^2(\omega))}^2 \leq c(h^{4+2r} + \tau^{2+2r} + \delta^2).$$
Moreover, Assumption \ref{assum: dis-reg} and the estimates \eqref{ineq:Ph stab}-\eqref{ineq:Ptau approx} yield
\begin{equation*}
   \widetilde{\gamma}_u(\|P_hP_\tau u^\dagger \|_{H^1(I;H^2(\Omega))}^2 + \|P_hP_\tau u^\dagger\|_{H^2(I;H^1(\Omega))}^2) \leq c\widetilde{\gamma}_u.
\end{equation*}
 The preceding inequalities and the assertions in Theorem \ref{lem:approx-uhs-fhs-Ihpih} yield the desired estimate.
\end{proof}

The following theorem states the error bounds on the minimizers $f_h^*$ and $\widetilde{f}_h^*$.
\begin{theorem}\label{thm:err-estimate}
Let Assumptions \ref{ass:reg} and \ref{assum: dis-reg} hold, and $f_h^*$ be the minimizer to problem \eqref{eqn:full-reg prob}. Then with $\eta^2:= \delta^2 + \tau^{2+2r} + h^{2+2r}$, there holds
\begin{equation*}
\|f^\dagger-f^*_h\|_{L^2(\Omega)} \leq c\eta.
 \end{equation*}
 Moreover, let $\widetilde f_h^*$ be the minimizer to problem \eqref{eqn:full-reg prob holder} and $\Omega_0\subset\subset\Omega$ be arbitrary. Then with $\widetilde\eta^2:=\widetilde{\gamma}_f + \widetilde{\gamma}_u + \delta^2 + \tau^{2+2r} + h^{2+2r}$, there exists $\theta\in (0,1)$, depending on $\Omega_0,\omega,t_0,\zeta,\Omega, R$ and the coefficients of $A$, such that
\begin{equation*}
\|f^\dagger-\widetilde f^*_h\|_{L^2(\Omega_0)} \leq c\widetilde\gamma_u^{-(1-\theta)/2}\widetilde\eta.
 \end{equation*}
\end{theorem}
\begin{proof}
 Let $G\equiv G(t) = Lu_{h,\tau}^* - Rf_h^*\in H^1(I;L^2(\Omega))$. Using  equation \eqref{eqn:gov}, we obtain the decomposition
\begin{equation*}
R(f^\dagger - f_h^*) = L(u^\dagger-u_{h,\tau}^*) + G.
\end{equation*}
By Theorem \ref{thm:Lip-perturb-stability} and Lemma \ref{lem:approx-uhs-fhs-Ihpih}, we arrive at
\begin{align*}
\|f^\dagger-f^*_h\|^2_{L^2(\Omega)} & \leq
c\big(\|\partial_t(u^\dagger-u_{h,\tau}^*)\|^2_{L^2(Q_\omega)} + \|A(u^\dagger-u_{h,\tau}^*)(t_0)\|^2_{L^2(\Omega)}
+ \|G\|^2_{H^1(I;L^2(\Omega))} \big)
\\& \leq c\big(\|\partial_t(u^\dagger-q)\|^2_{L^2(Q_\omega)} + \|Au^\dagger(t_0)-p\|^2_{L^2(\Omega)}  + J_{h,\tau}(u^*_{h,\tau},f^*_h)\big)\leq c\eta^2.
\end{align*}
This completes the first estimate. Next, let $\widetilde G\equiv \widetilde G(t) = L\widetilde u_{h,\tau}^* - R\widetilde f_h^*\in H^1(I;L^2(\Omega))$. Then there holds
\begin{equation*}
    R(f^\dagger - \widetilde f_h^*) = L(u^\dagger-\widetilde u_{h,\tau}^*) + \widetilde G.
\end{equation*}
Repeating the argument in Theorems \ref{thm:hold-stab} and \ref{thm:hold-perturb-stab} leads to the following estimate for some $\mu>0$ and all $s>0$
\begin{align*}
    \|f^\dagger - \widetilde f_h^*\|_{L^2(\Omega_0)} \leq &ce^{-c\mu s}M + ce^{cs}\big(\|A(u^\dagger-\widetilde u_{h,\tau}^*)(t_0) \|_{L^2(\Omega)} + \|\partial_t (u^\dagger - \widetilde u_{h,\tau}^*)\|_{L^2(I; H^1(\omega))} \\
    &+ \|\widetilde G\|_{H^1(I;L^2(\Omega))}\big)
    \leq ce^{-c\mu s}M + ce^{cs}\widetilde\eta.
\end{align*}	
with
$$M:= \|u^\dagger - \widetilde u_{h,\tau}^*\|_{H^1(I;H^2(\Omega))} + \|u^\dagger - \widetilde u_{h,\tau}^*\|_{H^2(I;H^1(\Omega))}.$$
Thus Lemma \ref{lem:approx-uhs-fhs-Ihpih-hol} and Assumption \ref{assum: dis-reg} imply the following uniform bound on $M$:
\begin{equation*}
    M \leq c(1 + \widetilde\gamma_u^{-\frac{1}{2}} \widetilde\eta)\leq c\widetilde\gamma_u^{-\frac{1}{2}}\widetilde\eta.
\end{equation*}
Balancing these two terms by a suitable choice of $s$ gives the desired estimate.
\end{proof}
\begin{remark}\label{remark:rate}
    Theorem \ref{thm:err-estimate} provides a priori error bounds on the discrete approximations $f_h^*$ and $\widetilde f_h^*$ (with / without the boundary condition of $u^\dag$), which depend explicitly on the noise level $\delta$, the penalty parameters $\bm \gamma = (\widetilde\gamma_f,\widetilde\gamma_u)$ and the discretization parameters $h$ and $\tau$. Thus it provides a useful guideline for the a priori choice of the algorithm parameters. When the penalty parameters are chosen as  $\widetilde\gamma_f=O(\delta^2)$ and $\widetilde\gamma_u=O(\delta^2)$, we take discretization parameters by $h = O(\delta^{\frac{2}{2+2r}})$ and $\tau = O(\delta^{\frac{2}{2+2r}})$ for $r=0,1$, and get the following estimate with respect to $\delta$
    \begin{equation*}
        \|f^\dag - f_h^*\|_{L^2(\Omega)} \leq c\delta \quad\mbox{and} \quad \|f^\dag - \widetilde f_h^*\|_{L^2(\Omega_0)} \leq c\delta^\theta.
    \end{equation*}
    These error bounds agree with the Lipschitz and H\"{o}lder stability estimates in Theorems \ref{thm:Lip-perturb-stability} and \ref{thm:hold-perturb-stab}. The stringent regularity requirement of the FE space in Theorem \ref{thm:err-estimate} is due to the nature of the conditional stability estimates; cf. Theorems \ref{thm:Lip-perturb-stability} and \ref{thm:hold-perturb-stab}.
\end{remark}

\begin{remark}
    Theorem \ref{thm:err-estimate} extends several existing works. Huhtala et al \cite{huhtala2014priori} and Chen et al \cite{chen2022stochastic}
    investigated the inverse problem of recovering the source from a finite number of density field measurements in elliptic and parabolic equations, respectively. Numerically, they employed the Galerkin FEM to discretize Tikhonov regularized model and proved an a priori error estimate of the discrete approximation. The analysis in both works relies on the well-definedness of the forward problem.
    When the initial / boundary conditions are unknown, Huang et al \cite[Theorem 2]{huang2020stability} established a novel H\"{o}lder conditional stability based on the given partial boundary data, but did not investigate the numerical issues.
\end{remark}

\section{Numerical experiments and discussions}\label{sec:numer}

Now we present numerical results of the reconstructed source $f$ with and without boundary data of $u$ separately using the observation data in \eqref{data: zero boundary} and \eqref{data: unknown boundary}, respectively. We investigate the convergence order of the approximations $f_h^*$ and $\widetilde f_h^*$, and the impact of the noise on the reconstructions. We discretize the state $u$ using Hermite elements and the source $f$ by continuous piecewise linear elements. The noise is added to all observation data pointwisely by $z^\delta = z^\dagger(1+\delta\xi(z))$, where $\xi(z)$ follows an i.i.d. standard normal distribution. The accuracy of the discrete approximation $f_h^*$ is measured by the $L^2(\Omega)$ error $e:=\|f^\dagger-f^*_h\|_{L^2(\Omega)}$.

\begin{example}\label{exam1}
Let $t_0=0.5$, $\Omega=(0,1)$, $I=(0.4,0.6)$ and $\omega=(0.2,0.8)$. Consider the elliptic operator $Av:=-v'\!'$ and the exact state $u^\dagger=\sin(\pi t)\sin(\pi x)$ with the source $F(x,t) \equiv R(t)f^\dag(x)=\big[\pi^2\sin(\pi t)+\pi\cos(\pi t)\big]\sin(\pi x)$.
\end{example}

We discretize the space variable using Hermite elements. The reconstruction errors $e$ for various mesh size $h$ and time step size $\tau$ are given in Table \ref{exam1:table}. When either the temporal and spatial meshes are simultaneously refined or only the spatial mesh is refined, the empirical rate  agrees well with the theoretical prediction in Theorem \ref{thm:err-estimate}. There is a small loss in the convergence order when the space-time mesh is very refined ($h=\frac{1}{120}, \tau=\frac{1}{120}$). This is caused by the bad conditioning of the discrete problem (due to the use of Hermite element). The convergence behavior in time is more complex: the spatial mesh has to be sufficiently fine in order to observe the convergence order in time ($h=\frac{1}{200}$), since the error in discretizing $f$ with linear elements dominates the time discretization error with Hermite elements. Moreover, there is a loss of the convergence order when the time mesh $\tau$ reaches $\frac{1}{30}$, again due to the bad conditioning of the discrete problem. Theorem \ref{thm:err-estimate} indicates a second-order convergence when $f$ and $u$ are smooth. Indeed, the error decomposition \eqref{ineq:error3term} shows that it is mainly governed by the term $\|L P_hP_\tau u^\dagger - R\pi_h f^\dagger\|_{H^1(I;L^2(\Omega))}^2$, which bounds the term $\|G\|_{H^1(I;L^2(\Omega))}$ (see the proof of Theorem \ref{thm:err-estimate}). This is numerically illustrated in Table \ref{exam1:table2}: $\|G\|_{H^1(I;L^2(\Omega))}$ indeed converges at the second order, but $\|G\|_{L^2(I;L^2(\Omega))}$ converges much faster (close to 4). Thus, an inverse type estimate might hold in a norm weaker than the $H^1(\Omega)$ for the PDE residual term.

\begin{table}[hbt!]
	\centering
\begin{threeparttable}
	\caption{\label{exam1:table} Example \ref{exam1}: The reconstruction error $e$ and convergence rates.}
	\begin{tabular}{cc|cc||cc|cc||cc|cc}
	\toprule
	$h$ & $\tau$ & $e$ & order & $h$ & $\tau$ & $e$ & order & $h$ & $\tau$ & $e$ & order\\
    \midrule
     $\frac{1}{10}$ & $\frac{1}{10}$ & 2.666e-3 & &$\frac{1}{10}$& $\frac{1}{100}$ & 2.629e-3 &&$\frac{1}{200}$&$\frac{1}{10}$&5.23e-4&\\\midrule
     $\frac{1}{20}$ & $\frac{1}{20}$ & 6.528e-4 & 2.030 &$\frac{1}{20}$&$ \frac{1}{100}$ &6.520e-4&2.011&$\frac{1}{200}$&$\frac{1}{20}$&3.43e-5&3.933\\\midrule
     $\frac{1}{40}$ & $\frac{1}{40}$ & 1.627e-4 & 2.005 &$\frac{1}{40}$&$\frac{1}{100}$&1.627e-4&2.003&$\frac{1}{200}$&$\frac{1}{30}$&7.75e-6&3.666\\\midrule
     $\frac{1}{80}$ & $\frac{1}{80}$ & 4.066e-5 & 2.000 &$\frac{1}{80}$&$\frac{1}{100}$&1.041e-4&2.001&$\frac{1}{200}$&$\frac{1}{40}$&8.83e-6&-0.453\\\midrule
     $\frac{1}{120}$ & $\frac{1}{120}$ & 1.896e-5 & 1.881 &$\frac{1}{120}$&$\frac{1}{100}$&4.065e-5&2.000&$\frac{1}{200}$&$\frac{1}{50}$&1.03e-5&-0.687\\
	\bottomrule
	\end{tabular}
\end{threeparttable}
\end{table}

\begin{table}[hbt!]
	\centering
\begin{threeparttable}
	\caption{\label{exam1:table2} Example \ref{exam1}: The convergence order of $ G$ in time under different norms.}
	\begin{tabular}{cc|c|c||c|c||c|c}
	\toprule
	$h$ & $\tau$ & $e$ & order & $\|G\|_{H^1(I;L^2(\Omega))}$ & order & $\|G\|_{L^2(I;L^2(\Omega))}$ & order\\
    \midrule
     $\frac{1}{200}$&$\frac{1}{10}$&5.23e-4&&5.692e-2&&4.056e-4&\\\midrule
     $\frac{1}{200}$&$\frac{1}{20}$&3.43e-5&3.933&1.395e-2&2.029&2.698e-5&3.910\\\midrule
     $\frac{1}{200}$&$\frac{1}{30}$&7.75e-6&3.666&6.154e-3&2.018&6.387e-6&3.554\\
	\bottomrule
	\end{tabular}
\end{threeparttable}
\end{table}

Next, we show the reconstructions without boundary information of $u$ using also the observation data $\partial_tu|_{\omega\times (t_0-\zeta, t_0+\zeta)}$ in Fig. \ref{exam1:fig:holder}. In the absence of the regularization, there are significant perturbations on the results as the noise level increases, which can be effectively mitigated by suitable regularization. In the experiments, we found the reconstruction is not very sensitive to $\widetilde{\gamma}_u$, so we set it a bit smaller.

\begin{figure}[hbt!]
\centering\setlength{\tabcolsep}{2pt}
	\begin{tabular}{cccc}
    \includegraphics[width=0.20\textwidth, clip]{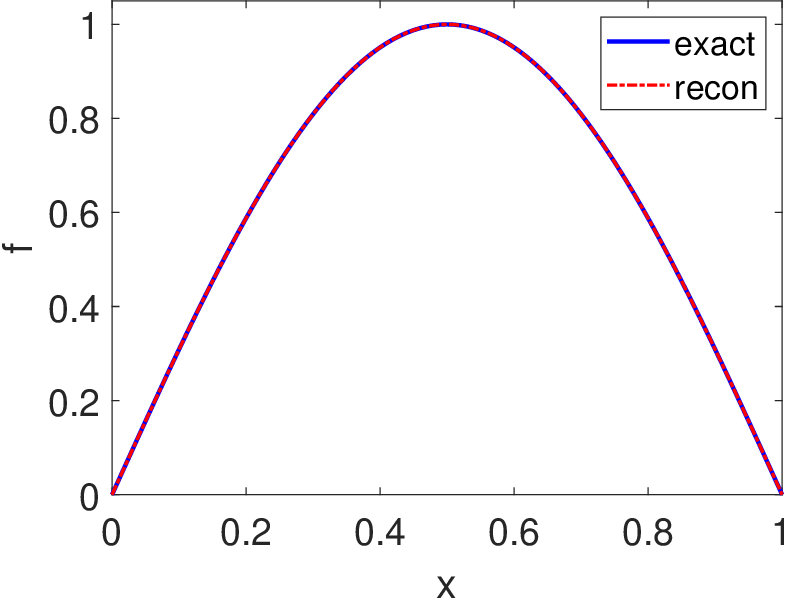}&
    \includegraphics[width=0.20\textwidth, clip]{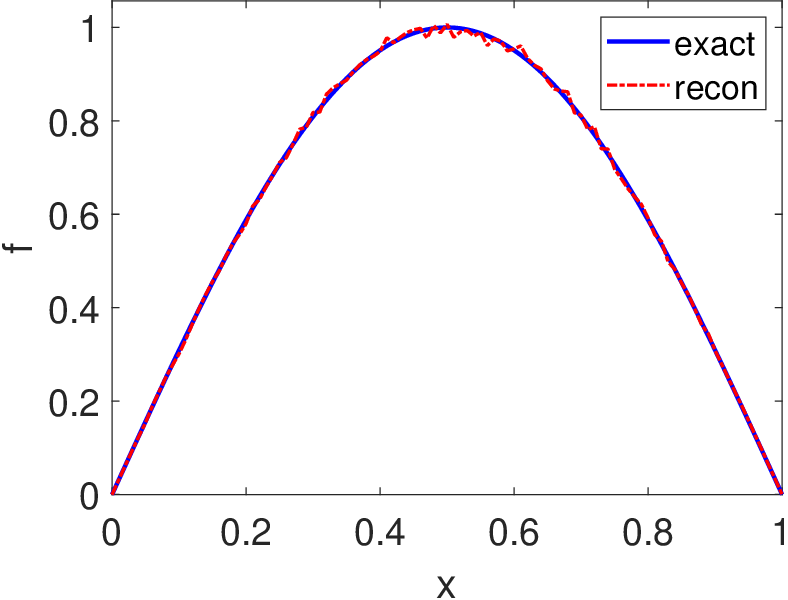}& \includegraphics[width=0.20\textwidth, clip]{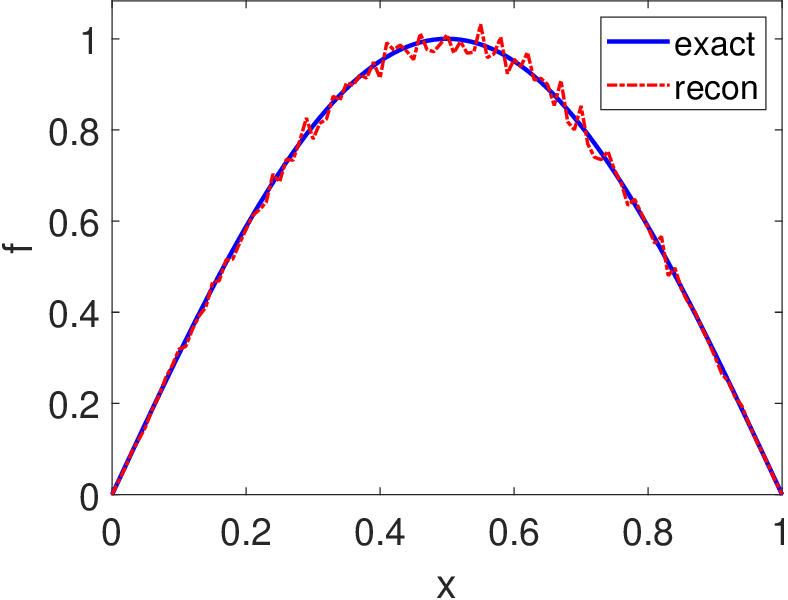}&
    \includegraphics[width=0.20\textwidth, clip]{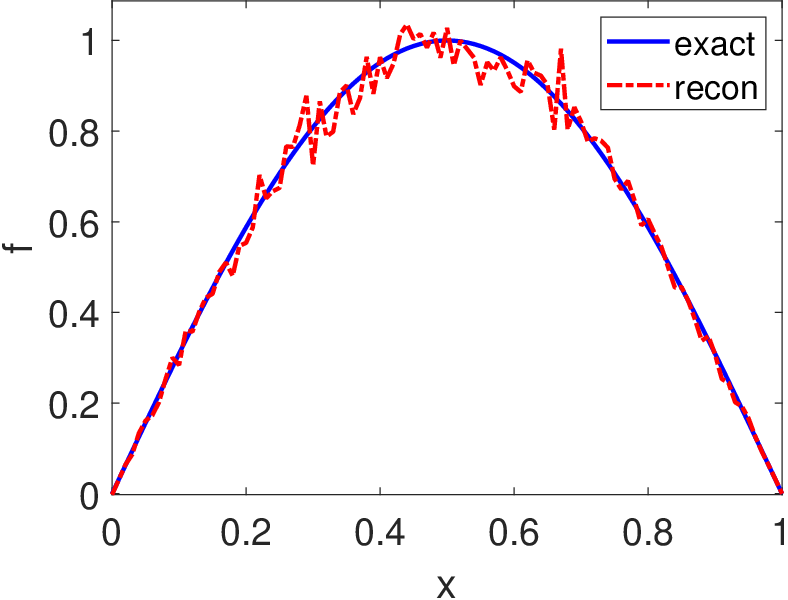}  \\
    &\includegraphics[width=0.20\textwidth, clip]{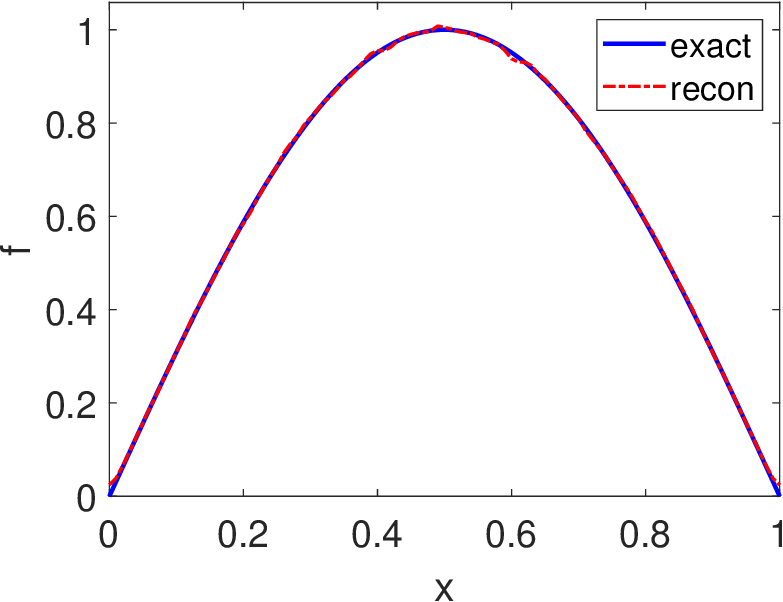}& \includegraphics[width=0.20\textwidth, clip]{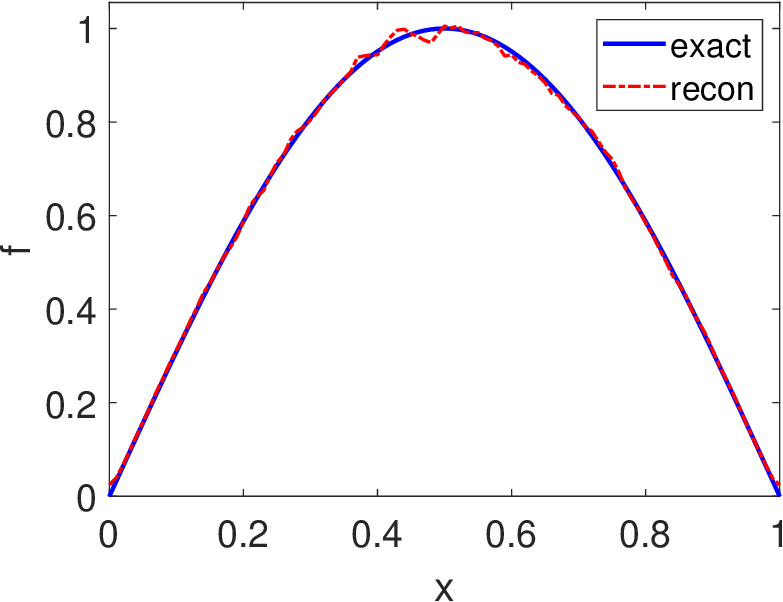}&
    \includegraphics[width=0.20\textwidth, clip]{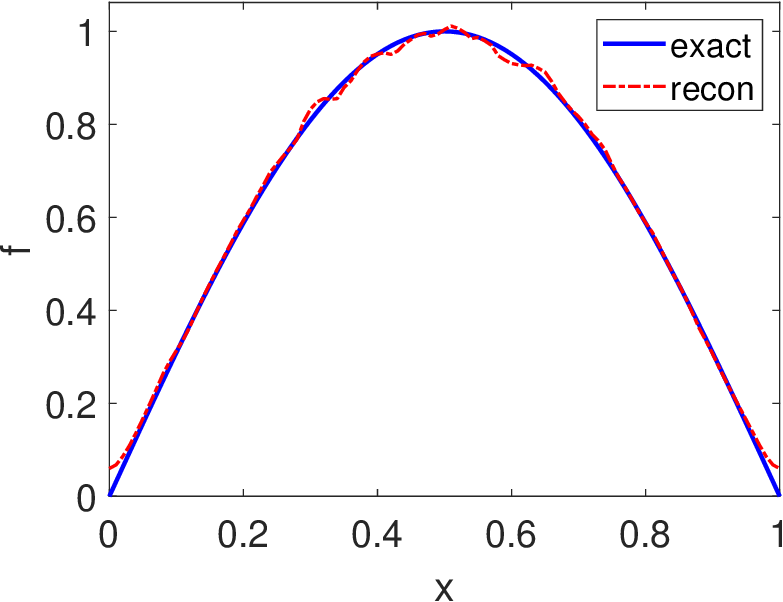} \\
    (a) $\delta=0\%$ & (b) \makecell{$\delta=2\%$\\($\widetilde\gamma_f=10^{-3}$, $\widetilde\gamma_u=10^{-4}$)} & (c) \makecell{$\delta=5\%$\\($\widetilde\gamma_f=10^{-3}$, $\widetilde\gamma_u=10^{-4}$)} & (d) \makecell{$\delta=10\%$\\($\widetilde\gamma_f=5\times 10^{-3}$, \\$\widetilde\gamma_u=5\times 10^{-4}$)}
\end{tabular}
\caption{\label{exam1:fig:holder} Example \ref{exam1}: The reconstructions without boundary information, at four noise levels without regularization (top) and with regularization (bottom). The numbers in the brackets are regularization parameters.}
\end{figure}

\begin{example}\label{exam2}
    Let $t_0=1.0$, $\Omega=(0,1)$, $I=(0.5,1.5)$ and $\omega=(0.2,0.8)$. Consider the elliptic operator $Av:=-v'\!'$, and the source is given by $f=\pi_hg$, with $g=\sin(\pi x)$.
\end{example}

The exact source $\pi_hg\in H^1(\Omega)$ but not in $H^2(\Omega)$, and the exact state $u$ is not available analytically. So we first generate the data by solving the forward problem on a finer space-time mesh. Since $f$ is explicitly approximated using linear elements, there is no approximation error for $\pi_h g$.
Due to the ill-posedness of the problem, the results are sensitive to the presence of data noise, and the convergence in the time direction is not steady but the empirical rate is still close to 4, cf. Table \ref{exam2:table}.

\begin{table}[htb!]
	\centering
    \begin{threeparttable}
	\caption{\label{exam2:table} Example \ref{exam2}: The reconstruction error $e$ and convergence order. $\operatorname{Cond}(A)$ denotes the condition number of the system matrix.}
	\begin{tabular}{c|ccccc}
	\toprule
	$h$  & $\frac{1}{20}$ & $\frac{1}{20}$  & $\frac{1}{20}$ & $\frac{1}{20}$ & $\frac{1}{20}$ \\
	$\tau$  & $\frac{1}{12}$ & $\frac{1}{18}$ & $\frac{1}{24}$ & $\frac{1}{36}$ & $\frac{1}{72}$\\
	\midrule
	$\operatorname{Cond}(A)$ & 7.113e12 & 1.878e13 & 3.999e13 & 1.212e14 & 7.743e14 \\
    \midrule
	$e$ ($\delta=0\%$) & 6.800e-4 & 1.464e-4 & 4.752e-5 & 9.731e-6 & 8.634e-7 \\
	\midrule
    order & & 3.787 & 3.912 & 3.911 & 3.495 \\
    \midrule
    $e$ ($\delta=0.2\%$) & 2.598e-1 & 1.384e-1 & 1.729e-1 & 1.120e-1 &2.806e-1 \\
    \midrule
    $e$ ($\delta=0.5\%$) & 6.367e-1 & 3.502e-1 & 3.801e-1 & 5.121e-1 & 6.758e-1 \\
    \midrule
    $e$ ($\delta=1.0\%$) & 1.037e0 & 8.476e-1 & 8.371e-1 & 1.131e0 & 1.167e0 \\
    \bottomrule
	\end{tabular}
\end{threeparttable}
\end{table}

\begin{example}\label{exam3}
Let $t_0=0.5$, $\Omega=(0,1)^2$, $I=(0.4,0.6)$ and $\omega=(0.2,0.8)^2$. Consider the elliptic operator $Av:=-\Delta v$ and the exact solution $u^\dagger=\sin(\pi t)\sin(\pi x)\sin(\pi y)$ with the source $F(x,y,t) \equiv R(t)f^\dag(x,y)=\big[2\pi^2\sin(\pi t)+\pi\cos(\pi t)\big]\sin(\pi x)\sin(\pi y)$.
\end{example}

We employ Argyris elements \cite{MASAYUKI1993381} to discretize the state $u$ in space. Argyris elements are not affine equivalent so the mesh must be a series of congruent triangles to facilitate the use of quadrature rules. We show the reconstructions with exact observation data in Fig. \ref{exam3:fig} with $h=\frac{1}{40}$ and $\tau=\frac{1}{10}$, which are fairly accurate, and investigate the convergence order in Table \ref{exam3:table}.

\begin{figure}[hbt!]
\centering\setlength{\tabcolsep}{2pt}
\begin{tabular}{ccc}
\includegraphics[width=0.3\textwidth, clip]{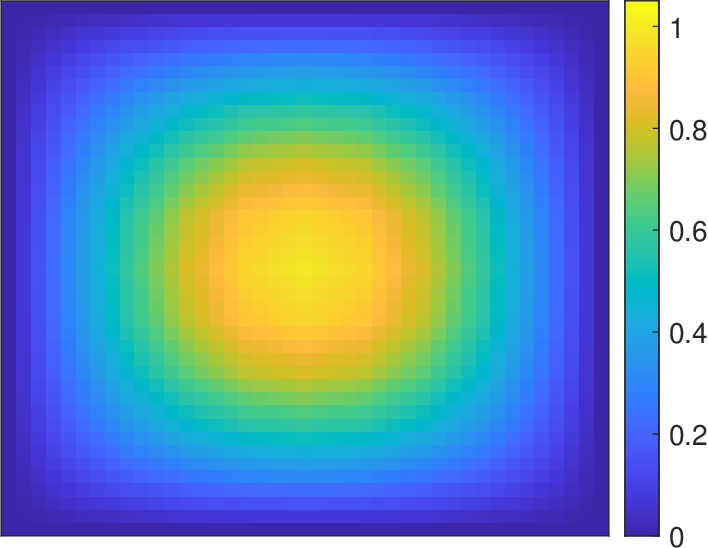} &
\includegraphics[width=0.3\textwidth, clip]{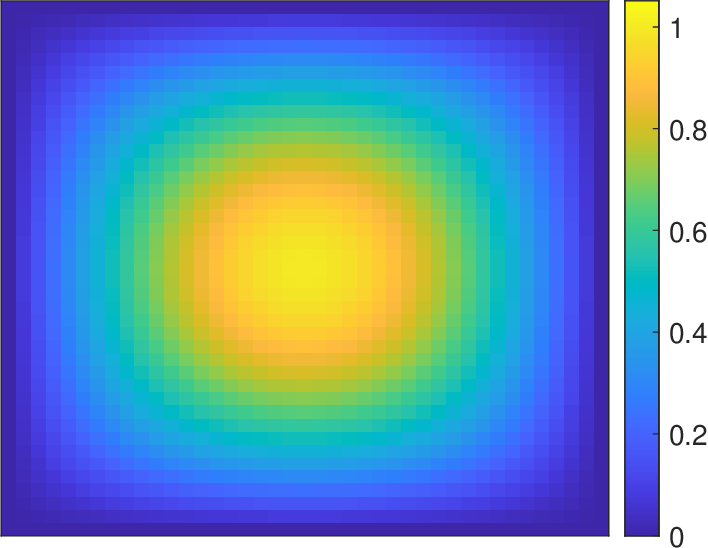} &
\includegraphics[width=0.3\textwidth, clip]{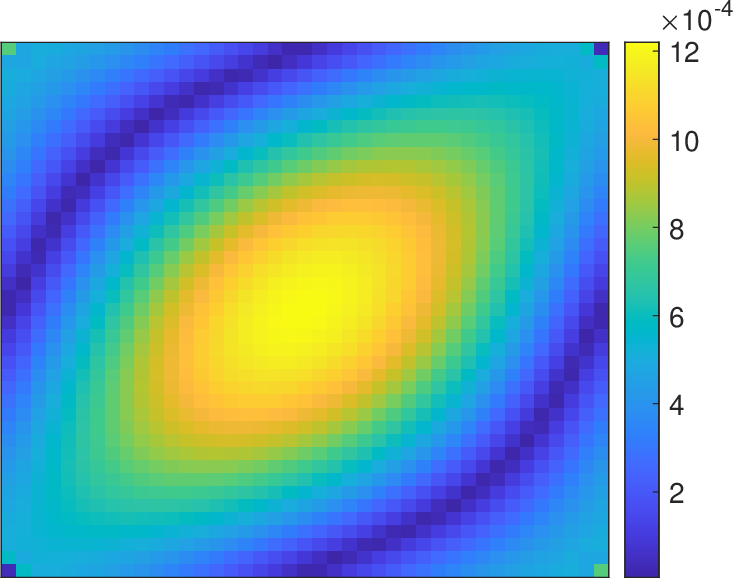}\\
exact & reconstruction & error \\
\end{tabular}
\caption{\label{exam3:fig}
Example \ref{exam3}: The reconstructed source $\hat f$  with exact data and boundary condition on $u$.}
\end{figure}

\begin{table}[htb!]
	\centering
\begin{threeparttable}
\caption{\label{exam3:table} Example \ref{exam3}: The reconstruction error $e$ and convergence order in time and space, with boundary information.}
	\begin{tabular}{c|cccc}
	\toprule
	$h$  & $\frac{1}{4}$ & $\frac{1}{8}$  & $\frac{1}{12}$ & $\frac{1}{16}$ \\
	$\tau$  & $\frac{1}{10}$ & $\frac{1}{20}$ & $\frac{1}{30}$ & $\frac{1}{40}$\\
	\midrule
	$e$ & 2.768e-2 & 6.593e-3 & 2.891e-3 & 1.618e-3 \\
	\midrule
    order & & 2.070 & 2.033& 2.018 \\
    \bottomrule
\end{tabular}
\end{threeparttable}	\end{table}

Next we show the reconstructions for Example \ref{exam3} without the boundary condition on $u$. The convergence behavior is similar to the 1-d case, but the example is much more ill-posed: with 1\% noise in the data, the reconstruction process completely fails. The regularization partly alleviates the issue, but the difference between the exact and recovered one is still large, cf. Fig. \ref{exam3:fig:holder}.

To verify the convergence rate of the discrete approximation $f_h^*$ (i.e., with the boundary condition of $u$) with respect to the noise level $\delta$, cf. Remark \ref{remark:rate}, we fix the time and spatial grids and vary the noise level $\delta$. The results in Fig. \ref{fig:deltaorder} indicate the first-order convergence, i.e., $O(\delta)$, which is consistent with the theoretical prediction in Theorem \ref{thm:err-estimate}. 

\begin{figure}[hbt!]
	\centering\setlength{\tabcolsep}{2pt}
	\begin{tabular}{cccc}
    \includegraphics[width=0.22\textwidth, clip]{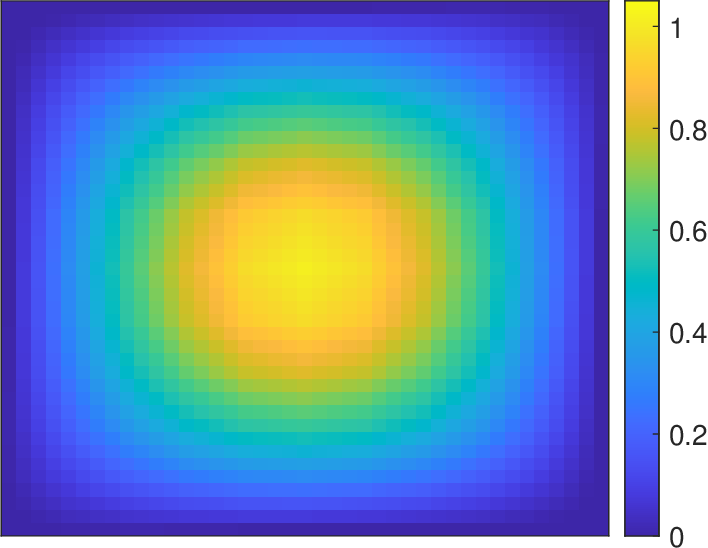}&
    \includegraphics[width=0.22\textwidth, clip]{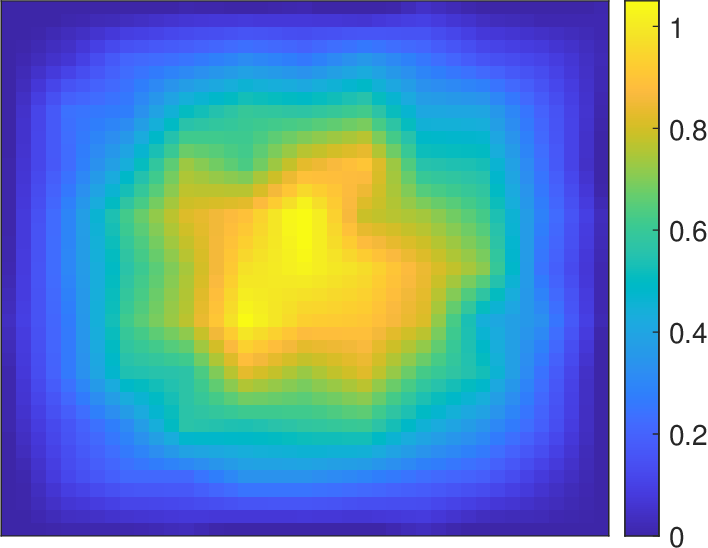}& \includegraphics[width=0.22\textwidth, clip]{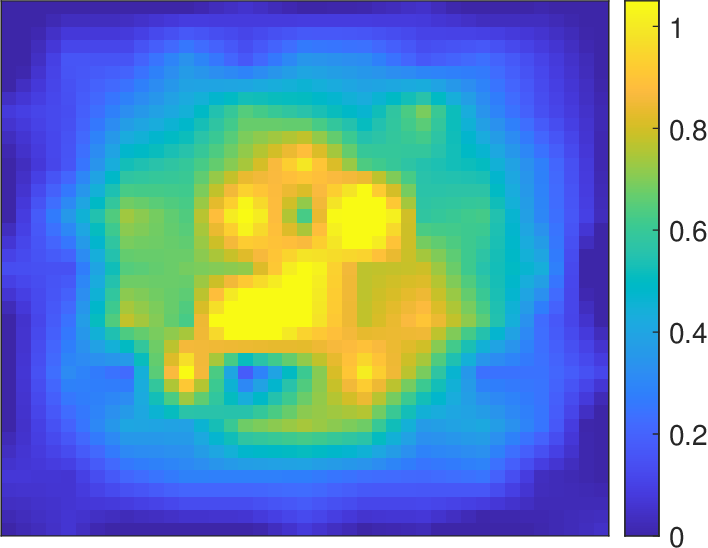}&
    \includegraphics[width=0.22\textwidth, clip]{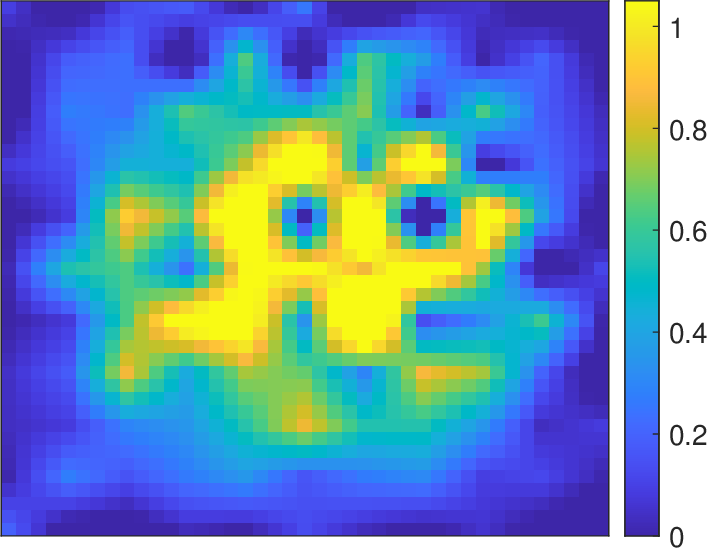}  \\
    \includegraphics[width=0.22\textwidth, clip]{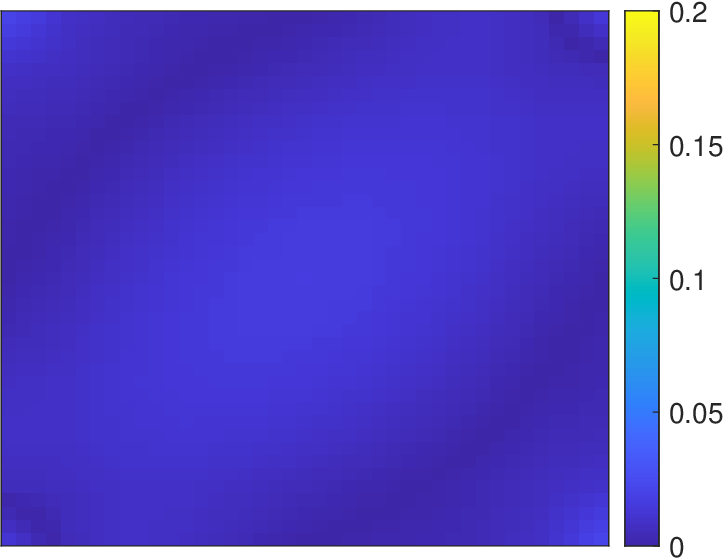}&
    \includegraphics[width=0.22\textwidth, clip]{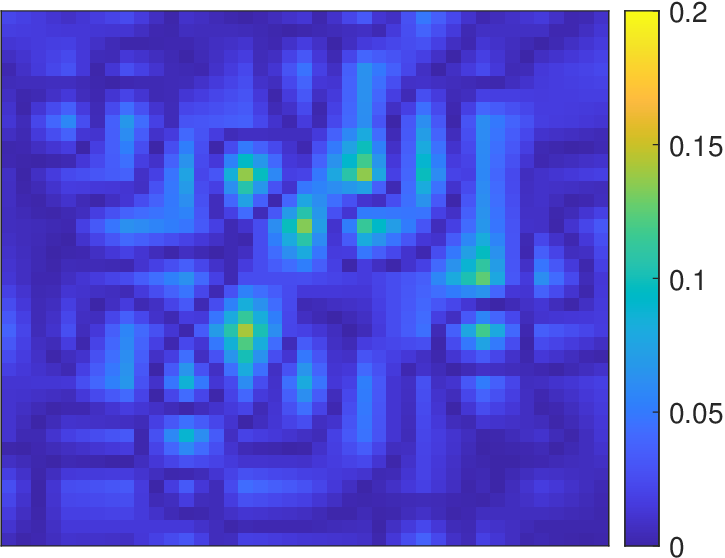}& \includegraphics[width=0.22\textwidth, clip]{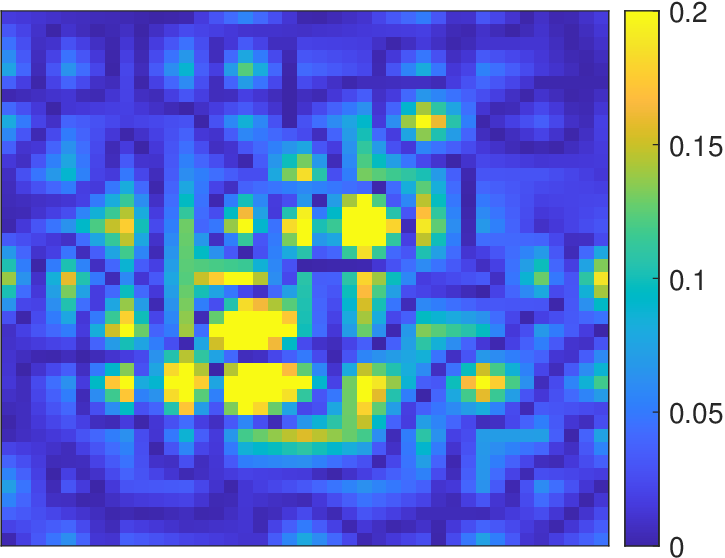}&
    \includegraphics[width=0.22\textwidth, clip]{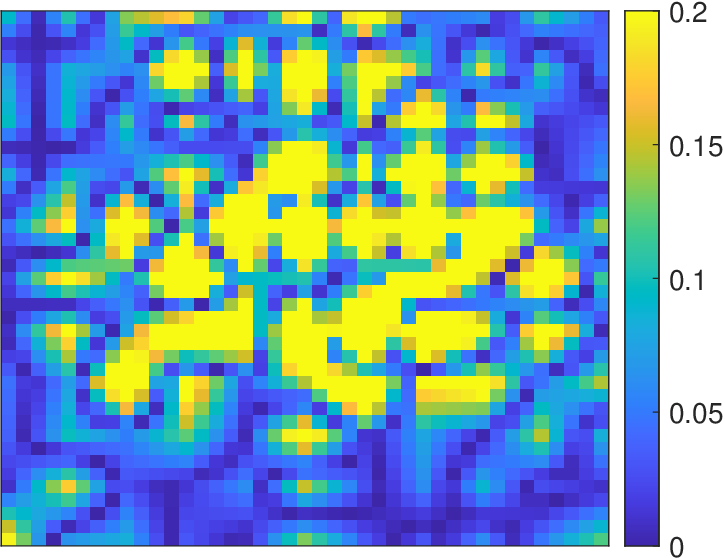} \\
    &\includegraphics[width=0.22\textwidth, clip]{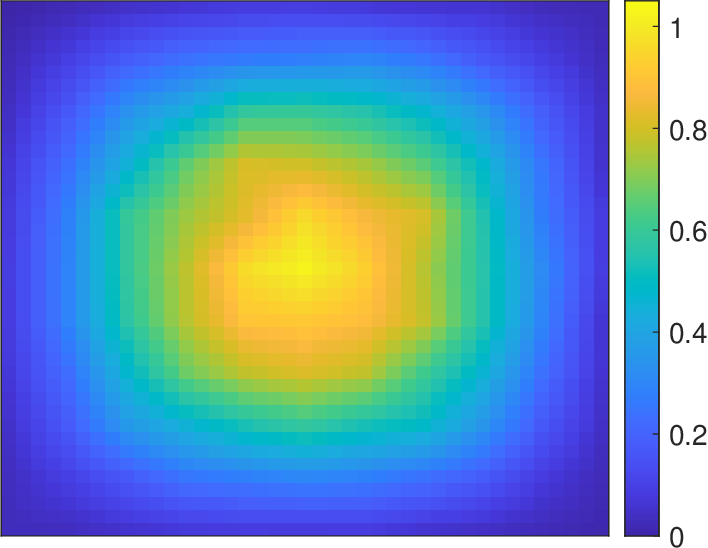}& \includegraphics[width=0.22\textwidth, clip]{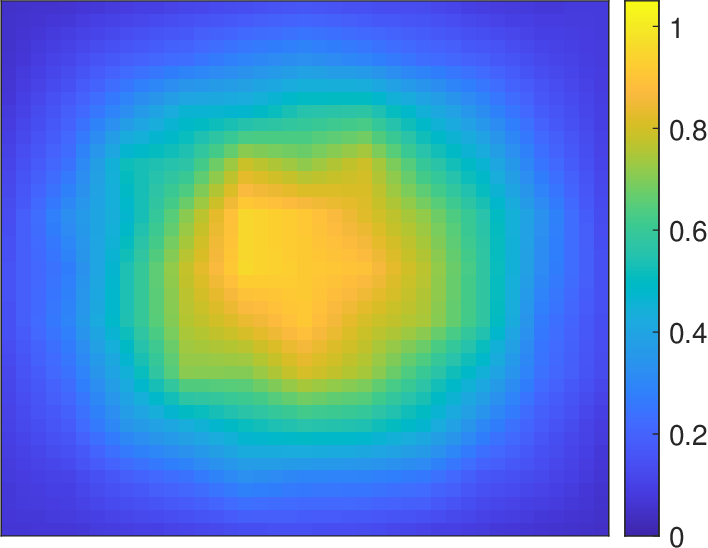}&
    \includegraphics[width=0.22\textwidth, clip]{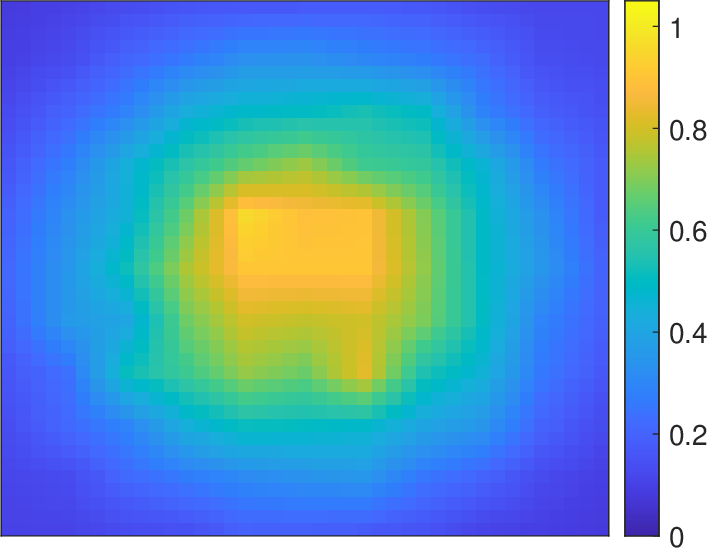} \\
    &\includegraphics[width=0.22\textwidth, clip]{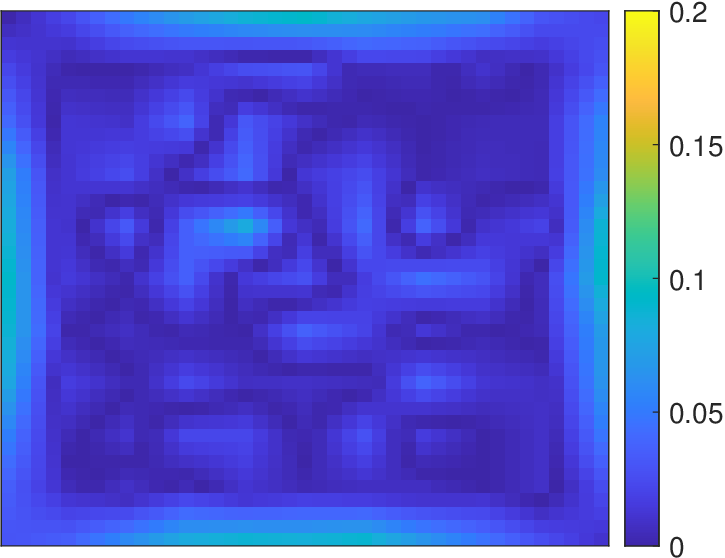}& \includegraphics[width=0.22\textwidth, clip]{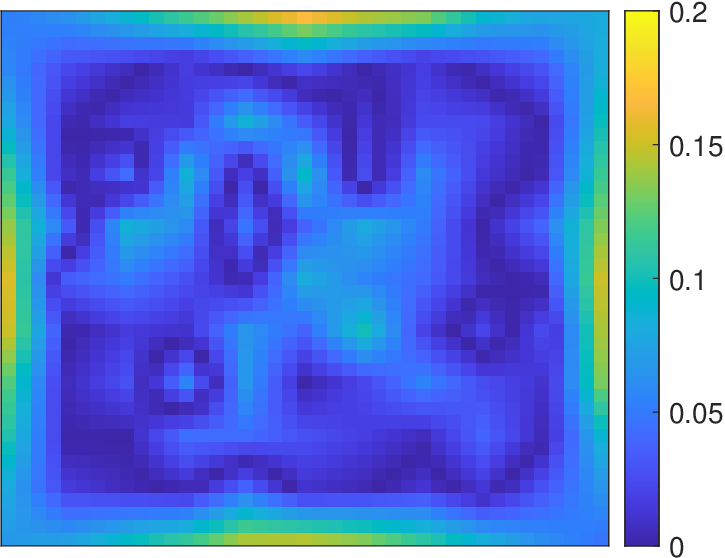}&
    \includegraphics[width=0.22\textwidth, clip]{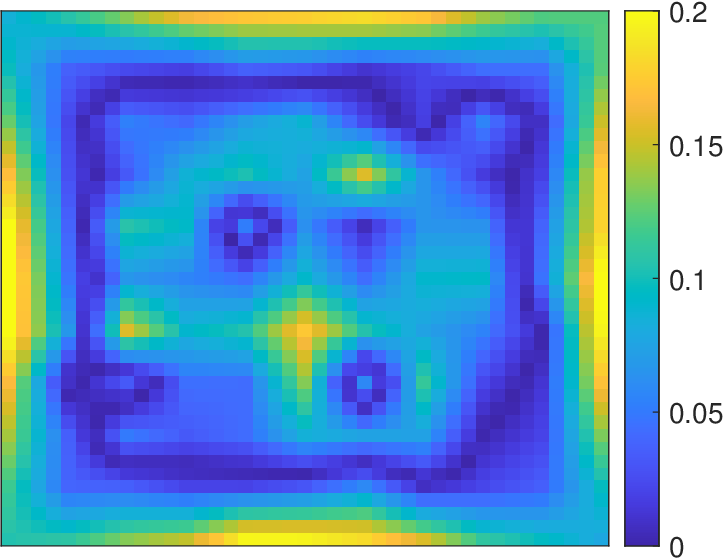} \\
    (a) $\delta=0\%$ & (b) \makecell{$\delta=0.2\%$\\($\widetilde\gamma_f=8\times10^{-2}$,\\$\widetilde\gamma_u=8\times10^{-4}$)} & (c) \makecell{$\delta =0.5\%$\\($\widetilde\gamma_f=2\times10^{-1}$,\\$\widetilde\gamma_u=2\times10^{-3}$)} & (d) \makecell{$\delta= 1.0\%$\\($\widetilde\gamma_f=4\times10^{-1}$,\\$\widetilde\gamma_u=4\times10^{-3}$)}
\end{tabular}
\caption{\label{exam3:fig:holder} Example \ref{exam3}: The reconstructions without boundary information, at four noise levels, without regularization (row 1) and their pointwise error (row 2), and the reconstruction with regularization (row 3) and their pointwise error (row 4).}
\end{figure}

\begin{figure}[hbt!]
	\centering\setlength{\tabcolsep}{2pt}
	\begin{tabular}{cc}
    \includegraphics[height=5cm]  {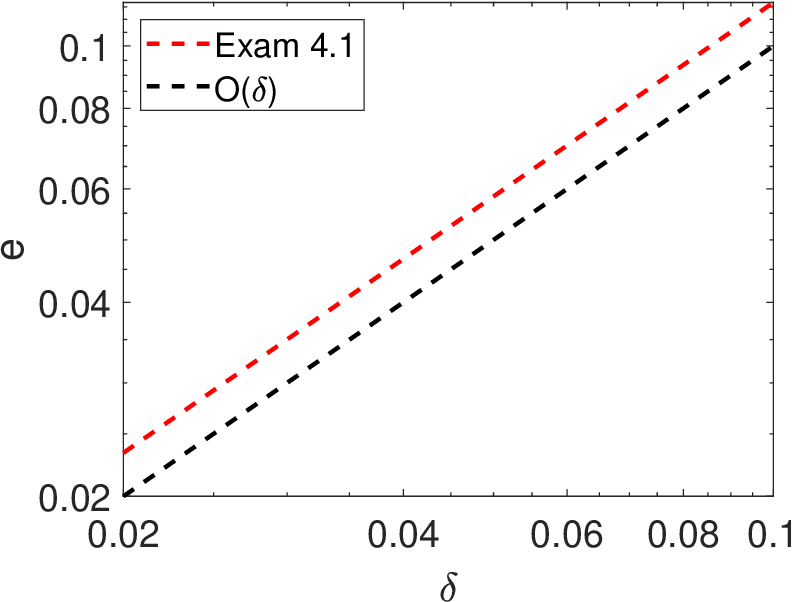}&
    \includegraphics[height=5cm]  {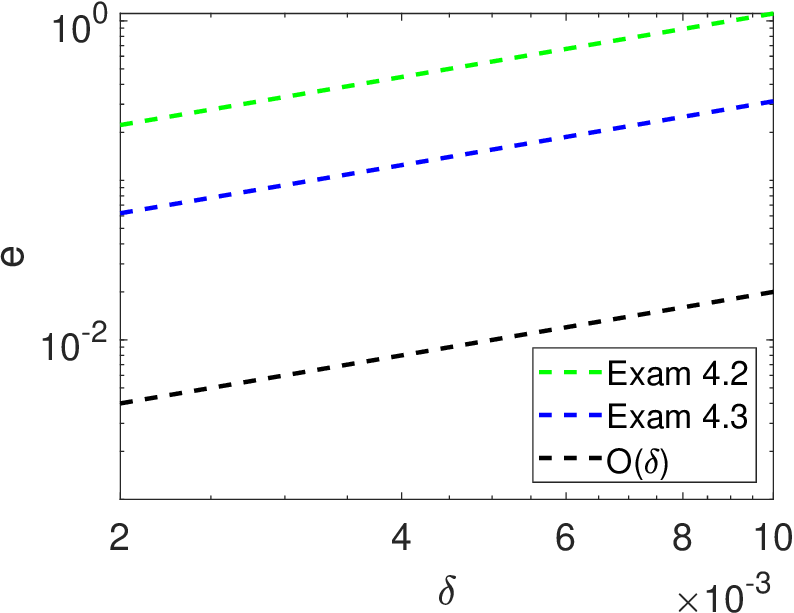}
	\end{tabular}
\caption{\label{fig:deltaorder}
The convergence order with respect to the noise level $\delta$ for the three examples, with  $h=\frac{1}{100},\tau=\frac{1}{50}$ for Example \ref{exam1}, $h=\frac{1}{20},\tau=\frac{1}{18}$ for Example \ref{exam2} and $h=\frac{1}{12},\tau=\frac{1}{20}$ for Example \ref{exam3}.}
\end{figure}

\section{Conclusion}
In this work, we have investigated an inverse source problem in a parabolic equation based on two types of observational data. Using Carleman estimates, we have established two new conditional stability estimates including the presence of source perturbation. Numerically, we have developed a regularized scheme, using the $H^2$ conforming FEM in both space and time to approximate the state and the usual P1 elements to approximate the source. Combined with the continuous stability estimates in the perturbation form, we derive error estimates and provide several numerical experiments to complement the theoretical analysis. The intricate structure of high-order conforming elements poses substantial challenges to practical implementation. In view of remarkable approximation capabilities of deep neural networks, one promising future direction is to employ neural networks to discretize the least-squares loss.

\appendix
\section{The proof of the estimate \eqref{estimate-w}}\label{append}

Following the argument in \cite{yamamoto2009carleman}, we give a detailed proof of the estimate \eqref{estimate-w}.
\begin{proof}
By the direct expansion of the elliptic operator $A$ and a perturbation argument, it suffices to prove that the estimate holds for the following operator in the non-divergence form for any $v\in C_0^\infty(Q)$
    \begin{equation*}
        L_0v:= \partial_t v - \sum_{i,j=1}^na_{ij}\partial^2_{x_ix_j}v = f, \quad \mbox{ in } Q.
    \end{equation*}
    Let $w = ve^{s\vartheta}$, and define an operator $P$ by
    \begin{equation}\label{equation-Pw}
        Pw = e^{s\vartheta}L_0(e^{-s\vartheta} w) = e^{s\vartheta}L_0v \equiv e^{s\vartheta}f.
    \end{equation}
    Then with $\sigma:=\sum_{i, j=1}^n a_{i j}\left(\partial_{x_i} d\right)\left(\partial_{x_j} d\right)$,  direct computation gives
\begin{equation}\label{expansion-Pw}
        \begin{split}
          P w& =\partial_t w - \sum_{i, j=1}^n a_{i j} \partial^2_{x_ix_j} w+2 s \lambda \vartheta \sum_{i, j=1}^n a_{i j}(\partial_{x_i} d)( \partial_{x_j} w) - s^2 \lambda^2 \vartheta^2 \sigma w \\
    & \quad +\Big[s \lambda^2 \vartheta \sigma + s \lambda \vartheta  \big(\sum_{i, j=1}^n a_{i j} \partial^2_{x_ix_j} d\big)-s \lambda \vartheta \left(\partial_t \psi\right)\Big]w \\
    & =:\partial_t w - \sum_{i, j=1}^n a_{i j} \partial^2_{x_i,x_j} w+2 s \lambda \vartheta \sum_{i, j=1}^n a_{i j}(\partial_{x_i} d) (\partial_{x_j} w) - s^2 \lambda^2 \vartheta^2 \sigma w + A_1 w.
        \end{split}
    \end{equation}
Next we decompose the function $Pw$ into two parts $ P_1w$ and $ P_2w$:
\begin{equation}\label{decom-opera-P}
\left\{
\begin{aligned}
 P_1w & = - \sum_{i, j=1}^n a_{i j} \partial^2_{x_i,x_j} w - s^2 \lambda^2 \vartheta^2 \sigma w + A_1 w,\\
 P_2w & = \partial_t w +2 s \lambda \vartheta \sum_{i, j=1}^n a_{i j}(\partial_{x_i} d) (\partial_{x_j} w).
\end{aligned}
\right.
\end{equation}
Then by the construction of $w$, we get
\begin{equation*}
     \int_Q |f|^2e^{2s\vartheta}\ \rmd x\rmd t = \int_Q|Pw|^2\ \rmd x\rmd t  \geq 2\int_QP_1wP_2w\ \rmd x\rmd t + \int_Q|P_2w|^2\ \rmd x\rmd t =: 2{\rm I} + {\rm II}.
\end{equation*}
We first consider the term ${\rm I}$. By the definitions of $P_1$ and $P_2$, it follows
\begin{align*}
{\rm I} &= {\rm I}_1 + {\rm I}_2 + {\rm I}_3 + {\rm I}_4 + {\rm I}_5 + {\rm I}_6.
\end{align*}
with the six terms $({\rm I}_i)_{i=1}^6$ given respectively by
\begin{align*}
 {\rm I}_1 & = - \sum_{i, j=1}^n \int_Q a_{i j} \partial^2_{x_i,x_j} w \partial_t w  \ \rmd x\rmd t ,\\
 {\rm I}_2 & = - \sum_{i, j=1}^n \int_Q a_{i j} \partial^2_{x_i,x_j} w 2 s \lambda \vartheta \sum_{k, \ell=1}^n a_{k\ell}(\partial_{x_k} d) (\partial_{x_\ell} w) \ \rmd x\rmd t, \\
{\rm I}_3 & = -\int_Q s^2 \lambda^2 \vartheta^2 \sigma w\partial_t w \ \rmd x\rmd t, \\
 {\rm I}_4 & = -\int_Q 2 s^3 \lambda^3 \vartheta^3 \sigma w \sum_{i, j=1}^n a_{i j}(\partial_{x_i} d)(\partial_{x_j} w)\ \rmd x\rmd t, \\
 {\rm I}_5 & =  \int_Q\left(A_1 w\right)\partial_t w\ \rmd x\rmd t, \\
 {\rm I}_6 & = \int_Q\left(A_1 w\right) 2 s \lambda \vartheta \sum_{i, j=1}^n a_{i j}(\partial_{x_i} d)(\partial_{x_j} w) \ \rmd x\rmd t.
\end{align*}
Now we bound the six terms separately. By integration by parts, the symmetry of the coefficient matrix $[a_{ij}]_{n\times n}$ and the Cauchy-Schwarz inequality, there holds for $w\in C_0^\infty(Q)$,
\begin{align*}
|{\rm I}_1| &= \Big|\sum_{i, j=1}^n \int_Q (\partial_{x_j}a_{i j}) \partial_{x_i} w \partial_t w  \ \rmd x\rmd t +  \sum_{i, j=1}^n \int_Q a_{i j} \partial_{x_i} w (\partial_{x_j}\partial_t w)  \ \rmd x\rmd t\Big| \\ &
= \Big|\sum_{i, j=1}^n \int_Q (\partial_{x_j}a_{i j}) \partial_{x_i} w \partial_t w  \ \rmd x\rmd t +  \sum_{i=1}^n \int_Q a_{i i} \partial_{x_i} w (\partial_{x_i}\partial_t w)  \ \rmd x\rmd t \\& \qquad + \sum_{\substack{i,j=1 \\ i > j}}^n \int_Q a_{i j} \Big[\partial_{x_i} w (\partial_{x_j}\partial_t w) + \partial_{x_j} w (\partial_{x_i}\partial_t w)\Big] \ \rmd x\rmd t\Big| \\
& = \Big|\sum_{i, j=1}^n \int_Q (\partial_{x_j}a_{i j}) \partial_{x_i} w \partial_t w  \ \rmd x\rmd t +  \frac12\sum_{i,j=1}^n \int_Q a_{i j} \partial_t(\partial_{x_j} w \partial_{x_i} w) \ \rmd x\rmd t\Big| \\
& = \Big|\sum_{i, j=1}^n \int_Q (\partial_{x_j}a_{i j}) \partial_{x_i} w \partial_t w  \ \rmd x\rmd t\Big| \leq c\int_{Q}s\vartheta\lambda|\nabla w|^2\ \rmd x\rmd t + c\int_Q \frac{1}{s\lambda\vartheta}|\partial_t w|^2\ \rmd x\rmd t.
\end{align*}
Similarly, using the basic estimates $|\partial_t\psi| = 2\beta|t-t_0| \leq c_{\beta,\zeta}$ and $|\sigma| \leq \mu^{-1} \|d\|^2_{C^1(\overline{\Omega})}\leq c_{\mu,d}$, we have
\begin{align*}
    |{\rm I}_3| & = \Big|\frac12\int_Q s^2 \lambda^2 \vartheta^2 \sigma \partial_t (w)^2 \ \rmd x\rmd t\Big| = \Big|\int_Q s^2 \lambda^2 \vartheta (\partial_t\vartheta) \sigma w^2 \ \rmd x\rmd t\Big| \\
    & = \Big|\int_Q s^2 \lambda^3 \vartheta^2 (\partial_t\psi) \sigma w^2 \ \rmd x\rmd t\Big| \leq c\int_Q s^2 \lambda^3 \vartheta^2 w^2 \ \rmd x\rmd t.
\end{align*}
Repeating the argument yields
\begin{align*}
    |{\rm I}_5| & = \Big|\frac12\int_Q A_1 \partial_t (w)^2 \ \rmd x\rmd t\Big| = \Big|\frac12\int_Q \partial_t A_1 w^2 \ \rmd x\rmd t\Big| \leq c\int_{Q}s\lambda^3\vartheta w^2 \ \rmd x\rmd t, \\
     |{\rm I}_6| & = \Big|\sum_{i, j=1}^n\int_QA_1 s \lambda \vartheta  a_{i j}(\partial_{x_i} d)\partial_{x_j} (w)^2 \ \rmd x\rmd t \Big| \\& = \Big|\sum_{i, j=1}^n\int_Q s \lambda\partial_{x_j}\Big[A_1  \vartheta  a_{i j}(\partial_{x_i} d)\Big] w^2 \ \rmd x\rmd t \Big| \leq c\int_Q s^2\lambda^4\vartheta^2 w^2 \ \rmd x\rmd t.
\end{align*}
For the term ${\rm I}_2$, the preceding argument gives
\begin{align*}
{\rm I}_2 & =   2 s \lambda\sum_{i, j=1}^n\sum_{k, \ell=1}^n \int_Q \partial_{x_i}\vartheta a_{i j} a_{k\ell}   (\partial_{x_k} d) ( \partial_{x_j} w) (\partial_{x_\ell} w) \ \rmd x\rmd t  \\& \quad +  2 s \lambda\sum_{i, j=1}^n\sum_{k, \ell=1}^n \int_Q \vartheta \partial_{x_i}\Big[a_{i j} a_{k\ell}(\partial_{x_k} d)\Big]    ( \partial_{x_j} w) (\partial_{x_\ell} w) \ \rmd x\rmd t \\&
\quad + 2 s \lambda\sum_{i, j=1}^n\sum_{k, \ell=1}^n \int_Q \vartheta a_{i j} a_{k\ell}(\partial_{x_k} d)    ( \partial_{x_j} w) (\partial^2_{x_ix_\ell} w) \ \rmd x\rmd t := {\rm I}_2^1 + {\rm I}_2^2 + {\rm I}_2^3.
\end{align*}
By H\"{o}lder's inequality, we have $|{\rm I}_2^2| \leq c \int_Q s \lambda\vartheta |\nabla w|^2 \ \rmd x\rmd t$. Meanwhile, by direct computation, we obtain
\begin{align*}
{\rm I}_2^1 & = 2 s \lambda^2\sum_{i, j=1}^n\sum_{k, \ell=1}^n \int_Q \vartheta  a_{i j} a_{k\ell}  (\partial_{x_i} d) (\partial_{x_k} d) ( \partial_{x_j} w) (\partial_{x_\ell} w) \ \rmd x\rmd t \\& = 2 s \lambda^2 \int_Q \vartheta  \Big|\sum_{i, j=1}^n a_{i j}   (\partial_{x_i} d)  ( \partial_{x_j} w) \Big|^2 \ \rmd x\rmd t \geq0.
\end{align*}
Similar to ${\rm I}_1$, the following identity holds for the term ${\rm I}_2^3$
\begin{align*}
{\rm I}_2^3 & =  s \lambda\sum_{i, j=1}^n\sum_{k, \ell=1}^n \int_Q \vartheta a_{i j} a_{k\ell}(\partial_{x_k} d)   \partial_{x_\ell}\Big[ ( \partial_{x_j} w) (\partial_{x_i} w)\Big] \ \rmd x\rmd t
\\& = -s \lambda\sum_{i, j=1}^n\sum_{k, \ell=1}^n \int_Q (\partial_{x_\ell}\vartheta) a_{i j} a_{k\ell}(\partial_{x_k} d)( \partial_{x_j} w) (\partial_{x_i} w) \rmd x \rmd t \\& \quad
 -s \lambda \sum_{i, j=1}^n \sum_{k, \ell=1}^n \int_Q \vartheta \partial_{x_\ell}\Big[a_{i j} a_{k \ell} (\partial_{x_k} d)\Big]( \partial_{x_j} w) (\partial_{x_i} w) \rmd x \rmd t  \\&
 =-s \lambda^2 \int_Q \vartheta \sigma \sum_{i, j=1}^n a_{i j}( \partial_{x_j} w) (\partial_{x_i} w) \rmd x \rmd t \\& \quad -s \lambda \sum_{i, j=1}^n \sum_{k, \ell=1}^n \int_Q \vartheta \partial_{x_\ell}\Big[a_{i j} a_{k \ell} (\partial_{x_k} d)\Big]( \partial_{x_j} w) (\partial_{x_i} w) \rmd x \rmd t.
\end{align*}
Thus we arrive at
\begin{equation*}
{\rm I}_2 \geq  - \int_Q  s \lambda^2\vartheta\sigma \sum_{i, j=1}^n a_{i j}( \partial_{x_j} w) (\partial_{x_i} w) \rmd x \rmd t - c \int_Q s \lambda\vartheta |\nabla w|^2 \ \rmd x\rmd t.
\end{equation*}
It remains to estimate ${\rm I}_4$. By integration by parts, we get
\begin{align*}
{\rm I}_4 & =  -\sum_{i, j=1}^n \int_Q  s^3 \lambda^3 \vartheta^3 \sigma   a_{i j}(\partial_{x_i} d)\partial_{x_j} (w^2)\ \rmd x\rmd t  \\
& = \sum_{i, j=1}^n \int_Q  3s^3 \lambda^3 \vartheta^2 \partial_{x_j}\vartheta \sigma   a_{i j}(\partial_{x_i} d)w^2\ \rmd x\rmd t + \sum_{i, j=1}^n \int_Q  s^3 \lambda^3 \vartheta^3 \partial_{x_j}\Big[\sigma   a_{i j}(\partial_{x_i} d)\Big] w^2\ \rmd x\rmd t \\ & = \int_Q  3s^3 \lambda^4 \vartheta^3 \sigma^2 w^2\ \rmd x\rmd t +  \sum_{i, j=1}^n \int_Q  s^3 \lambda^3 \vartheta^3 \partial_{x_j}\Big[\sigma   a_{i j}(\partial_{x_i} d)\Big] w^2\ \rmd x\rmd t \\& \geq \int_Q  3s^3 \lambda^4 \vartheta^3 \sigma^2 w^2\ \rmd x\rmd t - c\int_Q  s^3 \lambda^3 \vartheta^3 w^2\ \rmd x\rmd t.
\end{align*}
Summing the preceding estimates for ${\rm I}$ yields
\begin{align}
\label{appendix-estimate-I}
 & \int_Q  3s^3 \lambda^4 \vartheta^3 \sigma^2 |w|^2\ \rmd x\rmd t  - \int_Q  s \lambda^2\vartheta\sigma \sum_{i, j=1}^n a_{i j}( \partial_{x_j} w) (\partial_{x_i} w) \rmd x \rmd t  \\
 \leq & {\rm I}+ c\int_Q  \Big(s^3 \lambda^3 \vartheta^3 + s^2\lambda^4\vartheta^2\Big) |w|^2\ \rmd x\rmd t + c \int_Q s \lambda\vartheta |\nabla w|^2 \ \rmd x\rmd t + \frac{c}{\lambda}\int_Q\frac{1}{s\vartheta}|\partial_t w|^2\ \rmd x \rmd t. \nonumber
\end{align}
The definition of $P_2w$ and the basic inequality $|a+b|^2 \leq 2|a|^2 + 2|b|^2$ imply that for large $s>0$ such that $s\vartheta\geq1$
\begin{align}\label{appendix-estimate-II}     \epsilon\int_Q\frac{1}{s\vartheta}|\partial_t w|^2\ \rmd x \rmd t & \leq 2\epsilon\int_Q\frac{1}{s\vartheta}|P_2 w|^2\ \rmd x \rmd t + 2\epsilon\int_Q 4s \lambda^2 \vartheta \Big|\sum_{i, j=1}^n a_{i j}(\partial_{x_i} d) (\partial_{x_j} w)\Big|^2 \ \rmd x \rmd t \\& \leq
    \frac12 {\rm II} + c\epsilon\int_Q s \lambda^2 \vartheta |\nabla w|^2 \ \rmd x \rmd t,\nonumber
\end{align}
with $\epsilon\in(0,\frac14]$ to be determined. By the inequalities \eqref{appendix-estimate-I} and \eqref{appendix-estimate-II}, and the definitions of ${\rm I}$ and ${\rm II}$, we deduce
{\small\begin{align}\label{appendix-important-estimate}
       & \int_Q  3s^3 \lambda^4 \vartheta^3 \sigma^2 |w|^2\ \rmd x\rmd t   - \int_Q  s \lambda^2\vartheta\sigma \sum_{i, j=1}^n a_{i j}( \partial_{x_j} w) (\partial_{x_i} w) \rmd x \rmd t + (\epsilon-\frac{c}{\lambda})\int_Q\frac{1}{s\vartheta}|\partial_t w|^2\ \rmd x \rmd t   \\
       \leq &c\!\int_Q|f|^2e^{2s\vartheta} \rmd x \rmd t \!+\! c\!\int_Q  \Big(s^3 \lambda^3 \vartheta^3 \!+\! s^2\lambda^4\vartheta^2\Big) |w|^2 \rmd x\rmd t \!+\! c\! \int_Q s \lambda\vartheta |\nabla w|^2  \rmd x\rmd t \!+\! c\epsilon\!\int_Q s \lambda^2 \vartheta |\nabla w|^2 \rmd x \rmd t.\nonumber
    \end{align}}
Using \eqref{equation-Pw} and \eqref{expansion-Pw} and multiplying the test function $s \lambda^2 \vartheta \sigma w$ on both sides give
\begin{equation*}
    {\rm J}_1 + {\rm J}_2 + {\rm J}_3 + {\rm J}_4 + {\rm J}_5 = \int_Q fe^{s\vartheta}s \lambda^2 \vartheta \sigma w \ \rmd x\rmd t \leq c\int_Q |f|^2e^{2s\vartheta} \ \rmd x\rmd t + c\int_Q s^2 \lambda^4 \vartheta^2 |w|^2 \ \rmd x\rmd t,
\end{equation*}
with the terms $({\rm J}_i)_{i=1}^5$ given by, respectively,
\begin{align*}
 {\rm J}_1 & = \int_Q s \lambda^2 \vartheta \sigma (\partial_t w )w \ \rmd x\rmd t,
 && {\rm J}_2  = - \sum_{i, j=1}^n \int_Q  s \lambda^2 \vartheta  a_{i j} \sigma (\partial^2_{x_i,x_j} w)w \ \rmd x\rmd t,\\
{\rm J}_3 & =  \sum_{i, j=1}^n \int_Q 2 s^2 \lambda^3 \vartheta^2  a_{i j}(\partial_{x_i} d) \sigma (\partial_{x_j} w)w \ \rmd x\rmd t, &&
 {\rm J}_4 = -\int_Q  s^3 \lambda^4 \vartheta^3 \sigma^2 w^2  \ \rmd x\rmd t, \\
 {\rm J}_5 & = \int_Q   s \lambda^2 \vartheta A_1  \sigma w^2 \ \rmd x\rmd t.
\end{align*}
Using a similar argument and the definition of $A_1$, we deduce $|{\rm J}_5|\leq c\int_Q s^2 \lambda^4 \vartheta^2 |w|^2 \ \rmd x\rmd t$ and
\begin{align*}
|{\rm J}_1| & = \Big|\frac12\int_Q s \lambda^2 \vartheta \sigma \partial_t (w^2) \ \rmd x\rmd t\Big| = \Big|\frac12\int_Q s \lambda^2 \partial_t\vartheta \sigma  |w|^2 \ \rmd x\rmd t\Big| \leq     c \int_Q s \lambda^3 \vartheta  |w|^2 \ \rmd x\rmd t.
\end{align*}
Further, the bounds on ${\rm J}_3$ and ${\rm J}_2$ also hold
\begin{align*}
    |{\rm J}_3| & =  \Big|\sum_{i, j=1}^n \int_Q 2s^2 \lambda^3 \vartheta(\partial_{x_j}\vartheta) a_{i j}(\partial_{x_i} d) \sigma w ^2 \ \rmd x\rmd t \\ & \quad +\sum_{i, j=1}^n \int_Q s^2 \lambda^3 \vartheta^2  \partial_{x_j}\Big[a_{i j}(\partial_{x_i} d) \sigma\Big] w^2 \ \rmd x\rmd t\Big|
    \leq c\int_Q s^2\lambda^4\vartheta^2 |w|^2\ \rmd x\rmd t,
\end{align*}
and
\begin{align*}
    {\rm J}_2 & = \sum_{i, j=1}^n \int_Q  s \lambda^2 \vartheta  a_{i j} \sigma (\partial_{x_j} w)(\partial_{x_i}w) \ \rmd x\rmd t + \sum_{i, j=1}^n \int_Q  s \lambda^2 \partial_{x_i}\vartheta  a_{i j} \sigma (\partial_{x_j} w)w \ \rmd x\rmd t \\&\quad + \sum_{i, j=1}^n \int_Q  s \lambda^2 \vartheta  \partial_{x_i}\Big[a_{i j} \sigma\Big] (\partial_{x_j} w)w \ \rmd x\rmd t \\ &
    \geq  \int_Q  s \lambda^2 \vartheta \sigma  \sum_{i, j=1}^n a_{i j} (\partial_{x_j} w)(\partial_{x_i}w) \ \rmd x\rmd t - c\int_Q s \lambda^3\vartheta |\nabla w||w| \ \rmd x\rmd t \\&
    \geq \int_Q  s \lambda^2 \vartheta \sigma  \sum_{i, j=1}^n a_{i j} (\partial_{x_j} w)(\partial_{x_i}w) \ \rmd x\rmd t - c\int_Q s^2 \lambda^4\vartheta^2|w|^2 \ \rmd x\rmd t - c\int_Q \lambda^2 |\nabla w|^2 \ \rmd x\rmd t.
\end{align*}
Consequently,
\begin{equation}\label{appendix-important-estimate-plus}
    \begin{split}
     & \int_Q  s \lambda^2 \vartheta \sigma  \sum_{i, j=1}^n a_{i j} (\partial_{x_j} w)(\partial_{x_i}w) \ \rmd x\rmd t  -\int_Q  s^3 \lambda^4 \vartheta^3 \sigma^2 |w|^2  \ \rmd x\rmd t \\ \leq &c\int_Q |f|^2e^{2s\vartheta} \ \rmd x\rmd t + c\int_Q s^2 \lambda^4\vartheta^2|w|^2 \ \rmd x\rmd t + c\int_Q \lambda^2 |\nabla w|^2 \ \rmd x\rmd t.
    \end{split}
\end{equation}
Then the arguments for $\eqref{appendix-important-estimate} $ and $  \eqref{appendix-important-estimate-plus}$ give
{\footnotesize
\begin{align*}
& \int_Q  s \lambda^2 \vartheta \sigma  \sum_{i, j=1}^n a_{i j} (\partial_{x_j} w)(\partial_{x_i}w) \ \rmd x\rmd t  + \int_Q  s^3 \lambda^4 \vartheta^3 \sigma^2 |w|^2  \ \rmd x\rmd t  + (\epsilon-\frac{c}{\lambda})\int_Q\frac{1}{s\vartheta}|\partial_t w|^2\ \rmd x \rmd t \\
\leq &c\!\int_Q|f|^2e^{2s\vartheta} \rmd x \rmd t \!+\! c\!\int_Q  \Big(s^3 \lambda^3 \vartheta^3 \!+\! s^2\lambda^4\vartheta^2\Big) |w|^2 \rmd x\rmd t  \!+\! c\! \int_Q \Big(s \lambda\vartheta + \lambda^2\Big) |\nabla w|^2  \rmd x\rmd t  \!+\! c\epsilon\!\int_Q s \lambda^2 \vartheta |\nabla w|^2 \rmd x \rmd t.
\end{align*}
} This proves the desired estimate \eqref{estimate-w}.
\end{proof}

\bibliographystyle{siam}

\end{document}